\documentclass[12pt]{amsart}

\usepackage{amstext,amssymb,amsmath,stmaryrd,amsthm}
\usepackage{geometry} % see geometry.pdf on how to lay out the page.  
\usepackage{color}
\usepackage{hyperref}
\usepackage{enumitem}
\makeindex

\numberwithin{equation}{section}

\newtheorem{theorem}{Theorem}[section]
\newtheorem{corollary}[theorem]{Corollary}
\newtheorem{proposition}[theorem]{Proposition}

\newtheorem{remark}[theorem]{Remark}

\newtheorem{lemma}[theorem]{Lemma}

\begin{document}

\title[Elephant-Reinforced Neural Networks]
{Elephant Reinforced Galves-L\"ocherbach Networks}

\author{ Ioannis Papageorgiou }

 \thanks{\textit{Address:}  Universidade Federal do ABC (UFABC) - CMCC, Avenida dos Estados, 5001 - Santo Andre - Sao Paulo, Brasil.
\\ \text{\  \   \      } 
\textit{Email:}  i.papageorgiou@ufabc.edu.br, papyannis@yahoo.com  }


%\thanks{}


\begin{abstract}
We introduce an infinite-dimensional Galves--L\"ocherbach system with bounded
Elephant-type reinforced synaptic interactions. The reinforcement mechanism
modifies the probabilities of excitatory and inhibitory synaptic updates while
keeping their amplitudes uniformly bounded. We prove non-explosion by means of
a Lyapunov estimate and establish a conditional Wasserstein contraction for the
membrane-potential dynamics when the coupled systems share the same reinforcement
profile. Finally, we derive the corresponding replica mean-field equation under
the Poisson hypothesis.
\end{abstract}

\date{}
\maketitle

\section{Introduction}

Memory-dependent stochastic dynamics arise when the law of a future event is
modified by the accumulated history of the process. In neuronal systems, this
provides a natural abstract mechanism for representing synaptic adaptation:
previous excitatory and inhibitory transmissions may influence the probability
of future interactions. The purpose of this paper is to introduce complete-memory
reinforcement into an infinite-dimensional stochastic network of spiking
neurons.

Our reinforcement rule is inspired by the Elephant Random Walk, introduced by
Sch\"utz and Trimper \cite{SchutzTrimper2004}. In that model, a previous
increment is sampled from the whole trajectory and is either repeated or
reversed. This simple mechanism produces long-range dependence and several
distinct asymptotic regimes. Its relation with P\'olya-type urns was established
in \cite{BaurBertoin2016}, building on methods such as those of
\cite{ChauvinPouyanneSahnoun2011}. Martingale techniques and limit theorems were
developed in
\cite{Bercu2017,Laulin2022martingale,Coletti2017a,Coletti2017b,
ColettiPapageorgiou2019,Guevara2019}, while the superdiffusive regime was
analyzed further in
\cite{GuerinLaulinRaschel2023,GuerinLaulinRaschel2025}.

The Elephant Random Walk has generated numerous variants, including walks with
delays \cite{GutStadtmuellerDelays}, smooth memory decay \cite{Laulin2022},
stopping mechanisms \cite{Bercu2022stops}, multidimensional and tree-valued
models \cite{Qin2025,Mukherjee2025}, and systems displaying additional phase
transitions \cite{MaulikRoySadhukhan2025,Qin2025phase}. Estimation of the memory
parameter was considered in \cite{BercuLaulin2024}, and related reinforced
random walks were studied in
\cite{Bertenghi2021,BertenghiRosales2022,Qin2024}. Surveys and broader
treatments may be found in
\cite{GutStadtmuellerVariations,GutStadtmuellerReview,Laulin2022thesis}.
Recent Poincar\'e and exponential-convergence estimates for bounded
Elephant-type dynamics were obtained in
\cite{ColettiGrisiPapageorgiou2026}.

We transfer this complete-memory principle to the signs of synaptic
interactions. When neuron \(i\) spikes, neuron \(j\) receives an increment
\[
\xi_{ij}(N_t^j,S_t^j)\in\{+w_{ij},-v_{ij}\},
\]
where \(N_t^j\) records the number of reinforced updates received by neuron
\(j\), and \(S_t^j\) records their signed cumulative value. The amplitudes
\(w_{ij}\) and \(v_{ij}\) remain bounded, whereas the probabilities of the two
signs depend on the complete reinforcement history. Thus, unlike reinforcement
models in which the magnitude of the interaction grows, the present mechanism
changes only the distribution of the synaptic effect.

The neuronal component of the model belongs to the family of
piecewise-deterministic Markov processes introduced by Davis
\cite{Davis84,Davis93}. Between spikes, membrane potentials follow deterministic
flows; at spike times, the firing neuron is reset and the remaining coordinates
receive instantaneous synaptic inputs. Early stochastic models of biologically  neural networks, including
networks with both excitatory and inhibitory interactions, were studied by
Turova \cite{Turova1997Hourglass,Turova1997Stochastic} and contributed to the
development of the probabilistic framework for interacting spiking-neuron
systems. PDMP methods have been applied to
stochastic hybrid systems in biology and neuroscience in
\cite{C-D-M-R,PTW-10,ABGKZ,L17,Lo}. Functional inequalities for compact
degenerate pure-jump neuronal processes, including Poincar\'e-type and modified
log-Sobolev inequalities, were established in
\cite{HodaraPapageorgiou2019, HodaraPapageorgiou2022, Papageorgiou2020}.

Stochastic models of neuronal activity have a much longer history. Early
descriptions include the works of Stein, Tuckwell, Karpelevich, Malyshev, Rybko,
and Cottrell
\cite{Stein,Tuc,K-M-R,Cot92}. Network models accounting for synchronization,
oscillations, balanced excitation and inhibition, spontaneous activity, and
learning were developed in
\cite{Br-Ha99,Br2000,So-Vr96,Abe91,Am-Br,Am-Br97,A-E-V94,
Je-Tr-Wh95,Ma-La96}. These studies provide part of the broader motivation for
constructing analytically tractable infinite neuronal systems.

The specific neuronal framework considered here is the
Galves--L\"ocherbach model \cite{GalvesLocherbach2013}. In this model, every
neuron possesses a membrane potential and spikes with a state-dependent
intensity. A spike resets the firing coordinate and modifies other coordinates
through excitatory or inhibitory interactions. The model has been studied in
connection with phase transitions \cite{F-G-L}, metastability
\cite{L-M,romaro}, statistical estimation \cite{H-K-L,H-R-R},
hydrodynamic and mean-field limits
\cite{C-D-L-O,D-L,D-L-O,C17}, and infinite-component systems
\cite{H-L}.

In \cite{Papageorgiou2023}, infinite Galves--L\"ocherbach networks were studied
under assumptions extending beyond the classical Uniform Summability Principle.
Replica Mean Field limits for related systems with excitatory and inhibitory
activity were subsequently derived in \cite{Papageorgiou2025}, following the
RMF methodology introduced in
\cite{B-T,BaTa20}. The present paper adds a different feature: the conditional
distribution of a synaptic interaction is no longer fixed but evolves through
an Elephant-type memory mechanism.

To retain the Markov property, we enlarge the state space and consider
\[
Y_t=(X_t,N_t,S_t).
\]
The membrane-potential process \(X_t\) alone is history dependent, whereas the
lifted process is Markovian. The reinforcement variables change after synaptic
updates, but the jump amplitudes remain uniformly controlled. This boundedness
is essential for the infinite-dimensional estimates.

A complementary numerical analysis of the Elephant-reinforced neuronal model is carried out in \cite{ElephantSimulations}, including simulations of finite networks, parameter dependence, synchronization phenomena, and comparisons with the Replica Mean Field approximation. The present work focuses on the general analytical theory underlying these dynamics.

Our first result establishes non-explosion of the full endogenous
Elephant--Galves--L"ocherbach process. Under summability assumptions on the
firing-rate Lipschitz constants and on the maximal synaptic increments, a
weighted Lyapunov estimate controls the total firing intensity and the expected
number of spikes on every finite time interval. The argument is localized before
Dynkin's formula is applied, so that no prior assumption of global existence is
required.

We next derive exact first-moment identities for the reinforcement variables.
The evolution of the reinforcement counter is governed by the incoming firing
intensity, while the signed reinforcement variable satisfies a continuous-time
analogue of the classical Elephant random-walk drift relation. These identities
make explicit how the memory parameter enters the neuronal dynamics and also
show that the first-moment system is not closed in general, because of the
dependence between the reinforcement balance and the incoming network activity.

The next results concern conditional membrane-potential dynamics in prescribed
reinforcement environments. For a fixed reinforcement profile, we construct a
maximal coupling of two membrane-potential processes using common reinforced
signs at simultaneous spikes. Under suitable dissipativity and compatibility
conditions, this coupling yields exponential contraction in a weighted
$1$-Wasserstein distance. This is a quenched contraction result in a common
environment and does not imply contraction of the full endogenous process
$(X,N,S)$.

We then allow the two membrane-potential processes to evolve in different
prescribed reinforcement environments. By maximally coupling the corresponding
Elephant sign laws, we obtain a quantitative perturbation estimate in which the
Wasserstein distance is controlled by the initial discrepancy together with an
exponentially weighted measure of the difference between the two reinforcement
profiles. This provides a stability result with respect to perturbations of the
excitation--inhibition environment.

Finally, in a frozen reinforcement environment and under the Poisson
Hypothesis, we identify the conditional stationary Replica Mean Field identity
satisfied by the limiting membrane-potential transform. Incoming replica
interactions are replaced by independent Poisson inputs with self-consistent
rates. Since the reinforcement counters are nondecreasing, stationarity is
imposed only on the conditional membrane-potential dynamics. In the model on
$\mathbb R_+$, inhibitory jumps involving the positive-part map lead naturally
to a truncated exponential transform.

The remainder of the paper is organized as follows. We first define the bounded
Elephant--Galves--L"ocherbach model and describe the conditional drift induced
by the reinforcement mechanism. We then establish non-explosion of the full
endogenous process and derive exact moment identities for the reinforcement
variables. Next, we prove conditional Wasserstein contraction in a common
prescribed reinforcement environment and quantify the stability of the
membrane-potential dynamics under perturbations of that environment. The final
section derives the conditional stationary Replica Mean Field identity in a
frozen reinforcement environment under the Poisson Hypothesis.

\begin{remark} 
\label{rem:endogenous-conditional-distinction}

The results of Sections~\ref{sec:model}--\ref{subsec:reinforcement-moment-identities}
concern the full endogenous process
\begin{equation*}
Y_t=(X_t,N_t,S_t),
\end{equation*}
in which the reinforcement variables evolve jointly with the membrane
potentials.

By contrast, the Wasserstein results of
Sections~\ref{sec:wasserstein} and
\ref{sec:environment-stability} concern membrane-potential dynamics
conditioned on prescribed reinforcement environments. In particular, these
results do not imply contraction of the complete endogenous process
$(X_t,N_t,S_t)$.

Similarly, the Replica Mean Field analysis of
Section~\ref{sec:rmf-elephant} is performed in a frozen reinforcement
environment and under the Poisson Hypothesis. It therefore provides a
conditional stationary RMF identity for the membrane-potential dynamics,
rather than a convergence theorem for the full endogenous reinforced system.
\end{remark}

\section{The bounded Elephant--Galves--L\"ocherbach model}
\label{sec:model}

Let
$
Y_t=(X_t,N_t,S_t),
$
 where
$
X_t=(X_t^i)_{i\in\mathbb N}\in\mathbb R_+^{\mathbb N}
$
 is the membrane-potential configuration, while
\begin{equation*}
N_t=(N_t^i)_{i\in\mathbb N}\in\mathbb N_0^{\mathbb N},
\qquad
S_t=(S_t^i)_{i\in\mathbb N}\in\mathbb Z^{\mathbb N}
\end{equation*}
are the reinforcement variables. The variable $(N_t^i)$ counts the number of
reinforced synaptic updates received by neuron ($i$), and $(S_t^i)$ is the signed
sum of the corresponding Elephant variables.

We assume that the initial reinforcement variables satisfy
\begin{equation*}
N_0^i\in\mathbb N_0,
\qquad
S_0^i\in\mathbb Z,
\qquad
i\in\mathbb N,
\end{equation*}
together with the admissibility conditions
\begin{equation*}
|S_0^i|\leq N_0^i
\end{equation*}
and
\begin{equation*}
S_0^i\equiv N_0^i\pmod 2,
\qquad
i\in\mathbb N.
\end{equation*}
The particular choice
$N_0^i=S_0^i=0, 
i\in\mathbb N,
$ 
corresponds to the system with no reinforcement history before time zero. Whenever neuron $i$ receives a reinforced synaptic update, its reinforcement
variables change according to
\begin{equation*}
(N_t^i,S_t^i)
\longmapsto
(N_t^i+1,S_t^i+\varepsilon),
\qquad
\varepsilon\in\{-1,+1\}.
\end{equation*}

We now verify that the admissibility conditions are preserved by every
reinforced update. Suppose that, immediately before an update,
\begin{equation*}
|S_t^i|\leq N_t^i
\end{equation*}
and
\begin{equation*}
S_t^i\equiv N_t^i\pmod 2.
\end{equation*}
Then
\begin{align*}
|S_t^i+\varepsilon|
\leq
|S_t^i|+|\varepsilon|
\leq
N_t^i+1.
\end{align*}
Moreover,
\begin{align*}
(N_t^i+1)-(S_t^i+\varepsilon)
&=
N_t^i-S_t^i+1-\varepsilon.
\end{align*}
Since $N_t^i-S_t^i$ is even and
$1-\varepsilon\in\{0,2\},
$ 
the quantity
$(N_t^i+1)-(S_t^i+\varepsilon)
$ 
is also even. Hence the conditions
\begin{equation}\label{defN}
|S_t^i|\leq N_t^i
\text{
and} 
S_t^i\equiv N_t^i\pmod 2
\end{equation}
are preserved at every reinforced update. Therefore, if the initial state is
admissible, then for every time for which the process is defined as in (\ref{defN}).
   
 For every ordered pair of distinct neurons $(k,j\in\mathbb N)$, let
\begin{equation*}
w_{kj}\geq 0
\qquad\text{and}\qquad
v_{kj}\geq 0
\end{equation*}
denote, respectively, the excitatory amplitude and the magnitude of the
inhibitory amplitude of the synaptic interaction from neuron (k) to neuron
(j). Thus, when neuron (k) spikes, the synaptic increment received by neuron
(j) will take one of the two values
$
+w_{kj}
$
 or $
-v_{kj}.
$

For every presynaptic neuron (k), define its postsynaptic neighbourhood by
\begin{equation*}
\mathcal N_k
:=
\left\{
j\in\mathbb N\setminus\{k\}:
w_{kj}+v_{kj}>0
\right\}.
\end{equation*}
Thus, $(j\in\mathcal N_k)$
if and only if a spike of neuron ($k$) can produce a
nonzero synaptic effect on neuron ($j$).

  We assume throughout that
$\mathcal N_k$ is finite for every $k$.  Fix $p,q\in[0,1]$ with $p+q=1$.
For an admissible reinforcement state \((n,s)\), define the conditional law of
the Elephant sign \(\varepsilon_{ki}\in\{-1,+1\}\) by
\begin{equation*}
Q_i^{n,s}(+1)
:=
\mathbb P\bigl(\varepsilon_{ki}=+1\mid N^i=n_i,S^i=s_i\bigr)
=
\begin{cases}
\displaystyle
\frac12+\frac{(2p-1)s_i}{2n_i},&n_i\geq1,\\[2mm]
\displaystyle\frac12,&n_i=0,
\end{cases}
\end{equation*}
and
\begin{equation*}
Q_i^{n,s}(-1)
:=
\mathbb P\bigl(\varepsilon_{ki}=-1\mid N^i=n_i,S^i=s_i\bigr)
=
\begin{cases}
\displaystyle
\frac12-\frac{(2p-1)s_i}{2n_i},&n_i\geq1,\\[2mm]
\displaystyle\frac12,&n_i=0.
\end{cases}
\end{equation*}
The corresponding synaptic increment is
\begin{equation*}
\xi_{ki}
=
\begin{cases}
+w_{ki},&\varepsilon_{ki}=+1,\\
-v_{ki},&\varepsilon_{ki}=-1.
\end{cases}
\end{equation*}
The signs corresponding to distinct postsynaptic neurons are taken conditionally
independent.  This independence is not needed for the non-explosion estimate;
only their one-dimensional conditional laws and the uniform bound
\[
|\xi_{kj}|\leq c_{kj},
\qquad
c_{kj}:=\max\{w_{kj},v_{kj}\},
\]
will be used.

Let $\varepsilon_k=(\varepsilon_{kj})_{j\in\mathcal N_k}$ and define the
post-jump map
\[
\Delta_k^{\varepsilon_k}(x,n,s)
=
\left(
\Delta_k^{X,\varepsilon_k}(x),
\Delta_k^N(n),
\Delta_k^{S,\varepsilon_k}(s)
\right)
\]
coordinatewise by
\[
\bigl(\Delta_k^{X,\varepsilon_k}(x)\bigr)^j
=
\begin{cases}
0,&j=k,\\[1mm]
\bigl(x^j+\xi_{kj}(n_j,s_j,\varepsilon_{kj})\bigr)_+,
&j\in\mathcal N_k,\\[1mm]
x^j,&j\notin\mathcal N_k\cup\{k\},
\end{cases}
\]
\[
\bigl(\Delta_k^N(n)\bigr)^j
=
\begin{cases}
n_j+1,&j\in\mathcal N_k,\\
n_j,&j\notin\mathcal N_k,
\end{cases}
\]
and
\[
\bigl(\Delta_k^{S,\varepsilon_k}(s)\bigr)^j
=
\begin{cases}
s_j+\varepsilon_{kj},&j\in\mathcal N_k,\\
s_j,&j\notin\mathcal N_k.
\end{cases}
\]
Here and below $u_+:=\max\{u,0\}$.

Between spikes, the membrane potentials solve
\[
\frac{d}{dt}X_t^i=-a_i g(X_t^i),
\qquad i\in\mathbb N,
\]
while $N_t$ and $S_t$ remain constant.  Neuron $k$ spikes with intensity
$\phi_k(X_{t-}^k)$.  For every cylinder test function $f=f(x,n,s)$ that is
continuously differentiable in its membrane-potential variables, the generator
of the enlarged process is therefore
\begin{align}
\mathcal L f(x,n,s)
=&
-\sum_{i=1}^{\infty}
 a_i g(x^i)\,\partial_{x^i}f(x,n,s)
\nonumber\\
&+
\sum_{k=1}^{\infty}
\phi_k(x^k)
\left[
\mathbb E_{\varepsilon_k}^{n,s} f\bigl(\Delta_k^{\varepsilon_k}(x,n,s)\bigr)
-f(x,n,s)
\right].
\label{eq:elephant-generator}
\end{align}

\subsection{Conditional drift of the reinforced synaptic interactions}
\label{subsec:conditional-synaptic-drift}

The reinforcement mechanism acts on the conditional distribution of the sign of
each synaptic interaction.  We record here the resulting conditional drift and
its dependence on the Elephant parameter.

For an admissible reinforcement state $(n,s)$ and a neuron $j$, define
\[
r_j(n,s)
:=
\begin{cases}
\dfrac{s_j}{n_j},&n_j\geq1,\\
0,&n_j=0.
\end{cases}
\]
Then $|r_j(n,s)|\leq1$.  Recall that
\[
Q_j^{n,s}(+1)
=\frac12+\frac{2p-1}{2}r_j(n,s),
\]
and
\[
Q_j^{n,s}(-1)
=
\frac12-\frac{2p-1}{2}r_j(n,s).
\]

\begin{proposition} 
\label{prop:conditional-elephant-sign-drift}
Let $\varepsilon_{kj}\in\{-1,+1\}$ denote the reinforced sign generated when
neuron $k$ spikes and affects neuron $j\in\mathcal N_k$. Then
\[
\mathbb E\left[
\varepsilon_{kj}\mid X_{t-}=x,\ N_{t-}=n,\ S_{t-}=s
\right]
=\begin{cases} 
(2p-1)\dfrac{s_j}{n_j},& n_j\geq1, \\
 0, &n_j=0.
\end{cases}
\]
\end{proposition}

\begin{proof}
By the definition of the reinforced sign law,
\begin{align*}
\mathbb E\left[
\varepsilon_{kj}\mid x,n,s
\right]
&=
Q_j^{n,s}(+1)-Q_j^{n,s}(-1)
\\
&=
\left(
\frac12+\frac{2p-1}{2}r_j(n,s)
\right)
-
\left(
\frac12-\frac{2p-1}{2}r_j(n,s)
\right)
\\
&=
(2p-1)r_j(n,s).
\end{align*}
\end{proof}

The preceding identity concerns the sign variable.  The actual synaptic
increment received by neuron $j$ is
\[
\xi_{kj}
=\begin{cases}
+w_{kj},&\varepsilon_{kj}=+1,\\
-v_{kj},&\varepsilon_{kj}=-1.
\end{cases}
\]

\begin{proposition}
\label{prop:conditional-synaptic-increment}
For every $j\in\mathcal N_k$,
\begin{align}
\mathbb E\left[
\xi_{kj}
\mid X_{t-}=x,\ N_{t-}=n,\ S_{t-}=s
\right]
=
\frac{w_{kj}-v_{kj}}{2}
+
\frac{w_{kj}+v_{kj}}{2}
(2p-1)r_j(n,s).
\label{eq:conditional-synaptic-drift}
\end{align}
Equivalently, for $n_j\geq1$,
\[
\mathbb E\left[
\xi_{kj}\mid x,n,s
\right]
=
\frac{w_{kj}-v_{kj}}{2}
+
\frac{w_{kj}+v_{kj}}{2}
(2p-1)\frac{s_j}{n_j},
\]
whereas, for $n_j=0$,
\[
\mathbb E\left[
\xi_{kj}\mid x,n,s
\right]
=
\frac{w_{kj}-v_{kj}}{2}.
\]
\end{proposition}

\begin{proof}
Using the two possible values of $\xi_{kj}$, we obtain
\begin{align*}
\mathbb E\left[
\xi_{kj}\mid x,n,s
\right]
&=
w_{kj}Q_j^{n,s}(+1)
-
v_{kj}Q_j^{n,s}(-1)
\\
&=
w_{kj}
\left(
\frac12+\frac{2p-1}{2}r_j(n,s)
\right)
-
v_{kj}
\left(
\frac12-\frac{2p-1}{2}r_j(n,s)
\right)
\\
&=
\frac{w_{kj}-v_{kj}}{2}
+
\frac{w_{kj}+v_{kj}}{2}
(2p-1)r_j(n,s).
\end{align*}
\end{proof}

\begin{corollary}
\label{cor:symmetric-synaptic-amplitudes}
Suppose that
$w_{kj}=v_{kj}=c_{kj}.
$ 
Then
\[
\mathbb E\left[
\xi_{kj}\mid x,n,s
\right]
=
c_{kj}(2p-1)r_j(n,s).
\]
Consequently, in the symmetric-amplitude case the conditional synaptic drift is
entirely generated by the reinforcement mechanism.
\end{corollary}

\begin{remark} 
\label{rem:elephant-parameter-interpretation}
Formula \eqref{eq:conditional-synaptic-drift} separates the intrinsic asymmetry
of the synaptic amplitudes from the contribution of the reinforcement memory.

\begin{enumerate}[label=\textnormal{(\roman*)}]
\item If $p=\frac12$, then
$Q_j^{n,s}(+1)=Q_j^{n,s}(-1)=\frac12,
$ 
and the reinforcement history has no effect on the next sign. In this case,
\[
\mathbb E[\xi_{kj}\mid x,n,s]
=\frac{w_{kj}-v_{kj}}{2}.
\]
\item If $p>\frac12$, then $2p-1>0$. A positive accumulated sign
$s_j>0$ increases the probability of a future excitatory interaction, whereas a
negative accumulated sign increases the probability of a future inhibitory
interaction. Thus the reinforcement favors persistence of the accumulated sign.

\item If $p<\frac12$, then $2p-1<0$. A positive accumulated sign increases the
probability of an inhibitory interaction, while a negative accumulated sign
increases the probability of an excitatory interaction. Thus the reinforcement
favors reversal of the accumulated sign.

\item If $w_{kj}\neq v_{kj}$, the first term
$
\frac{w_{kj}-v_{kj}}{2}
$
 represents a baseline excitatory--inhibitory asymmetry, while the second term
 $
\frac{w_{kj}+v_{kj}}{2}(2p-1)r_j(n,s)
$ 
is the history-dependent contribution.
\end{enumerate}
\end{remark}

\begin{remark} 
Since $|r_j(n,s)|\leq1$, Proposition~\ref{prop:conditional-synaptic-increment}
implies
\[
\left|
\mathbb E[\xi_{kj}\mid x,n,s]
\right|
\leq
\frac{|w_{kj}-v_{kj}|}{2}
+
\frac{w_{kj}+v_{kj}}{2}|2p-1|
\leq
c_{kj},
\]
where
$ c_{kj}:=\max\{w_{kj},v_{kj}\}.
$ 
Thus the conditional drift remains uniformly bounded, consistently with the
bounded-increment assumptions used in the non-explosion argument.
\end{remark}

\section{Non-explosion}
\label{sec:nonexplosion}

Set $\ell_i:=\|\phi_i\|_{\mathrm{Lip}}$ and assume
\begin{equation}
\sum_{i=1}^{\infty}\ell_i<\infty,
\qquad
\sum_{i=1}^{\infty}\phi_i(0)<\infty.
\label{eq:L-assumption}
\end{equation}
We also assume
\begin{equation}
C_*
:=
\sup_{k\in\mathbb N}
\sum_{j\in\mathcal N_k}\ell_j c_{kj}
<\infty.
\label{eq:bounded-influence}
\end{equation}
Finally, $g:\mathbb R_+\to\mathbb R_+$ is locally Lipschitz, $g(0)=0$, and
$a_i\geq0$ for every $i$. Thus the deterministic flow is dissipative in the
sense that
\[
-a_i g(x)\leq0,
\qquad x\geq0.
\]

Define
\[
h(x):=\sum_{i=1}^{\infty}\ell_i x^i,
\qquad
\mathcal X_h:=\{x\in\mathbb R_+^{\mathbb N}:h(x)<\infty\}.
\]
Because the reinforcement variables affect only the probabilities of the two
bounded synaptic increments, they need not appear in the Lyapunov function.

Let
 $\varepsilon_k
=
(\varepsilon_{kj})_{j\in\mathcal N_k}$. Conditionally on the reinforcement state \((n,s)\), the coordinates
\(\varepsilon_{kj}\), \(j\in\mathcal N_k\), are independent and satisfy
\begin{equation*}
\mathbb P\bigl(\varepsilon_{kj}=\sigma\mid n,s\bigr)
=
Q_j^{n,s}(\sigma),
\qquad
\sigma\in\{-1,+1\}.
\end{equation*}
We denote by
$\mathbb E_{\varepsilon_k}^{n,s}$ the 
expectation with respect to this conditional distribution of
\(\varepsilon_k\).
\begin{lemma}
\label{lem:elephant-generator-bound}

Let
\[
C_0
:=
\sum_{i=1}^{\infty}\phi_i(0),
\]
and, for $L\in\mathbb N$, define
\[
h_L(x)
:=
\sum_{i=1}^{L}\ell_i x^i.
\]

Then, for every $x\in\mathcal X_h$, every admissible $(n,s)$, and every
$L\in\mathbb N$,
\[
\sum_{k=1}^{\infty}
\phi_k(x^k)
\mathbb E_{\varepsilon_k}^{n,s}
\left[
h_L\bigl(\Delta_k^{X,\varepsilon_k}(x)\bigr)-h_L(x)
\right]
\leq
C_*\bigl(C_0+h(x)\bigr).
\]
Moreover, on every sublevel set $\{h<m\}$,
\begin{align}
\sum_{k=1}^{\infty}
\phi_k(x^k)
\mathbb E_{\varepsilon_k}^{n,s}
\left|
h\bigl(\Delta_k^{X,\varepsilon_k}(x)\bigr)-h(x)
\right|
<\infty,
\label{eq:absolute-h-jump-bound}
\end{align}
uniformly in the admissible reinforcement state $(n,s)$.

The same bound holds uniformly with $h$ replaced by $h_L$,
$L\in\mathbb N$.

\end{lemma}

\begin{proof}
Fix $L\in\mathbb N$. If neuron $k$ spikes, then
\begin{align*}h_L\bigl(\Delta_k^{X,\varepsilon_k}(x)\bigr)-h_L(x)
=
-\mathbf 1_{\{k\leq L\}}\ell_kx^k
+
\sum_{\substack{j\in\mathcal N_k\\ j\leq L}}
\ell_j
\left[
(x^j+\xi_{kj})_+-x^j
\right].
\end{align*}

Since $x^j\geq0$ and $|\xi_{kj}|\leq c_{kj}$,
\[
\left|
(x^j+\xi_{kj})_+-x^j
\right|
\leq
c_{kj}.
\]
Consequently,
\begin{align*}
\mathbb E_{\varepsilon_k}^{n,s}
\left[
h_L\bigl(\Delta_k^{X,\varepsilon_k}(x)\bigr)-h_L(x)
\right]
&\quad\leq
-\mathbf 1_{\{k\leq L\}}\ell_kx^k
+
\sum_{\substack{j\in\mathcal N_k\\ j\leq L}}
\ell_jc_{kj}
\\
&\quad\leq
-\mathbf 1_{\{k\leq L\}}\ell_kx^k
+
\sum_{j\in\mathcal N_k}\ell_jc_{kj}.
\end{align*}

Discarding the non-positive reset contribution gives
\[
\sum_{k=1}^{\infty}
\phi_k(x^k)
\mathbb E_{\varepsilon_k}^{n,s}
\left[
h_L\bigl(\Delta_k^{X,\varepsilon_k}(x)\bigr)-h_L(x)
\right]
\leq
C_*
\sum_{k=1}^{\infty}\phi_k(x^k).
\]
By the Lipschitz bound
\[
\phi_k(x^k)
\leq
\phi_k(0)+\ell_kx^k,
\]
we have
\[
\sum_{k=1}^{\infty}\phi_k(x^k)
\leq
C_0+h(x).
\]
Therefore,
\[
\sum_{k=1}^{\infty}
\phi_k(x^k)
\mathbb E_{\varepsilon_k}^{n,s}
\left[
h_L\bigl(\Delta_k^{X,\varepsilon_k}(x)\bigr)-h_L(x)
\right]
\leq
C_*\bigl(C_0+h(x)\bigr).
\]
We now prove the absolute local bound. For every $k$,
\begin{align*}
\left|
h_L\bigl(\Delta_k^{X,\varepsilon_k}(x)\bigr)-h_L(x)
\right|
\quad\leq
\ell_kx^k
+
\sum_{j\in\mathcal N_k}\ell_jc_{kj}.
\end{align*}
Set $u_k:=\ell_kx^k$. Then, 
$\sum_{k=1}^{\infty}u_k<m$, on the sublevel set $\{h<m\}$. 
Hence
\begin{align*}
\sum_{k=1}^{\infty}
\phi_k(x^k)
\mathbb E_{\varepsilon_k}^{n,s}
\left|
h_L\bigl(\Delta_k^{X,\varepsilon_k}(x)\bigr)-h_L(x)
\right|
\leq
C_0m+m^2+C_*C_0+C_*m.
\end{align*}
The right-hand side is independent of $L$.

Moreover, since $\mathcal N_k$ is finite,
\begin{align*}
\left|
h\bigl(\Delta_k^{X,\varepsilon_k}(x)\bigr)-h(x)
\right|
\leq
\ell_kx^k
+
\sum_{j\in\mathcal N_k}
\ell_jc_{kj}.
\end{align*}
Therefore, putting again $u_k:=\ell_kx^k$,
\begin{align*}\sum_{k=1}^{\infty}
\phi_k(x^k)
\mathbb E_{\varepsilon_k}^{n,s}
\left|
h\bigl(\Delta_k^{X,\varepsilon_k}(x)\bigr)-h(x)
\right|
\leq
\sum_{k=1}^{\infty}
\bigl(\phi_k(0)+u_k\bigr)
\left(
u_k+C_*
\right).
\end{align*}
On $\{h<m\}$, $\sum_k u_k<m$, and hence
\[
\sum_{k=1}^{\infty}
\phi_k(x^k)
\mathbb E_{\varepsilon_k}^{n,s}
\left|
h\bigl(\Delta_k^{X,\varepsilon_k}(x)\bigr)-h(x)
\right|
\leq
C_0m+m^2+C_*C_0+C_*m.
\]
This proves both absolute local bounds.
\end{proof}

\begin{proposition}  
\label{prop:local-construction}

Let $Y_0=(X_0,N_0,S_0)$ with $X_0\in\mathcal X_h$. Under
\eqref{eq:L-assumption}, the total firing intensity is finite at every state
$x\in\mathcal X_h$. More precisely,
\begin{equation*}
\sum_{k=1}^{\infty}\phi_k(x^k)
\leq
C_0+h(x)
<
\infty,
\end{equation*}
where
\begin{equation*}
C_0
:=
\sum_{k=1}^{\infty}\phi_k(0).
\end{equation*}

Consequently, starting from any state
$(x,n,s)$ with $x\in\mathcal X_h$, the first jump time is well defined.
After a spike of neuron $k$, the post-jump membrane-potential configuration
remains in $\mathcal X_h$. Indeed,
\begin{align*}
h\bigl(\Delta_k^{X,\varepsilon_k}(x)\bigr)
\leq
h(x)
+
\sum_{j\in\mathcal N_k}
\ell_j c_{kj}
\leq
h(x)+C_*.
\end{align*}
Hence the process may be constructed recursively from one jump time to the
next as long as the jump times do not accumulate.

Therefore there exists a unique minimal c\`adl\`ag process
\begin{equation*}
Y_t=(X_t,N_t,S_t),
\qquad
0\leq t<\tau_\infty,
\end{equation*}
where $\tau_\infty$ denotes the accumulation time of the successive spike
times. The process is unique up to $\tau_\infty$.
\end{proposition}

\begin{proof}
For every $k\in\mathbb N$, the Lipschitz property of $\phi_k$ gives
\begin{equation*}
\phi_k(x^k)
\leq
\phi_k(0)+\ell_k x^k.
\end{equation*}
Summing over $k$ yields
\begin{align*}
\sum_{k=1}^{\infty}\phi_k(x^k)
\leq
\sum_{k=1}^{\infty}\phi_k(0)
+
\sum_{k=1}^{\infty}\ell_k x^k
=
C_0+h(x)
<
\infty.
\end{align*}
Thus, at every state in $\mathcal X_h$, the superposition of the neuronal
firing clocks has finite total intensity.

Moreover, after a spike of neuron $k$,
\begin{align*}
h\bigl(\Delta_k^{X,\varepsilon_k}(x)\bigr)-h(x)
=
-\ell_k x^k
+
\sum_{j\in\mathcal N_k}
\ell_j
\left[
(x^j+\xi_{kj})_+-x^j
\right]
\leq
\sum_{j\in\mathcal N_k}
\ell_j c_{kj}
\leq
C_*.
\end{align*}
Since $h(x)<\infty$, it follows that
\begin{equation*}
h\bigl(\Delta_k^{X,\varepsilon_k}(x)\bigr)<\infty.
\end{equation*}
Thus every individual jump maps $\mathcal X_h$ into itself.

Between successive jumps, the deterministic flow satisfies
\begin{equation*}
\frac{d}{dt}X_t^i=-a_i g(X_t^i),
\end{equation*}
with $a_i\geq0$ and $g\geq0$, so the membrane potentials do not increase
along the deterministic flow. In particular, the deterministic evolution
preserves $\mathcal X_h$.

Therefore the usual recursive construction of a pure-jump PDMP applies:
starting from $Y_0$, one constructs the first jump time using the finite total
intensity, evolves deterministically until that time, applies the corresponding
post-jump map, and repeats the procedure. This construction is well defined
until the successive jump times accumulate. Its accumulation time is denoted
by $\tau_\infty$.

Since the deterministic flow, firing intensities, reinforcement laws, and
post-jump maps are prescribed by the current state, the recursive construction
is pathwise unique up to $\tau_\infty$.
\end{proof}

 Let  $N_k^{\mathrm{sp}}[0,t]$ denote the number of spikes produced by neuron
$k$ during the time interval $[0,t]$. We define the total number of spikes
in the network up to time $t$ by
\begin{equation*}
N^{\mathrm{sp}}[0,t]
:=
\sum_{k=1}^{\infty}N_k^{\mathrm{sp}}[0,t].
\end{equation*} Let
\[
\rho_r:=\inf\{t\geq0:N^{\mathrm{sp}}[0,t]\geq r\},
\qquad
\tau_\infty:=\lim_{r\to\infty}\rho_r,
\]
and, for $m,r\in\mathbb N$, define
\[
\eta_m:=\inf\{t\geq0:h(X_t)\geq m\},
\qquad
\sigma_{m,r}:=\eta_m\wedge\rho_r.
\]

For every $k\in\mathbb N$, let
$\mathcal E_k
:=
\{-1,+1\}^{\mathcal N_k}.
$ Since $\mathcal N_k$ is finite, write
\begin{equation*}
\mathcal N_k
=
\{j_1,\ldots,j_{m_k}\},
\qquad
m_k:=|\mathcal N_k|.
\end{equation*}

Let
$\Pi^k(ds,dz,du)
$ be independent Poisson random measures on
$\mathbb R_+\times\mathbb R_+
\times[0,1]^{m_k}
$ with intensity
$ds\,dz\,du,
$ where $du$ denotes Lebesgue measure on $[0,1]^{m_k}$.

For an admissible reinforcement state $(n,s)$ and
$u=(u_1,\ldots,u_{m_k})\in[0,1]^{m_k}$, define the sign vector
\begin{equation*}
F_k(n,s,u)
=
\left(
F_{k,j}(n,s,u)
\right)_{j\in\mathcal N_k}
\in\mathcal E_k
\end{equation*}
by
\begin{equation*}
F_{k,j_r}(n,s,u)
=
\begin{cases}
+1,
&
u_r\leq Q_{j_r}^{n,s}(+1),
\\
-1,
&
u_r>Q_{j_r}^{n,s}(+1),
\end{cases}
\qquad
r=1,\ldots,m_k.
\end{equation*}
Then, conditionally on $(N_{s-},S_{s-})$, the coordinates of
$F_k(N_{s-},S_{s-},u)
$ are independent and have respective laws
$Q_j^{N_{s-},S_{s-}}$, $j\in\mathcal N_k$.

The marked point measure associated with the spikes of neuron $k$ is defined
as the image of $\Pi^k$ under the predictable marking map
\begin{equation*}
u
\longmapsto
F_k(N_{s-},S_{s-},u).
\end{equation*}
We denote this integer-valued random measure by
$N^k(ds,dz,d\varepsilon_k).
$ Its predictable compensator is
\begin{equation*}
ds\,dz\,
Q_k^{N_{s-},S_{s-}}(d\varepsilon_k),
\end{equation*}
where
\begin{equation*}
Q_k^{n,s}(d\varepsilon_k)
=
\bigotimes_{j\in\mathcal N_k}
Q_j^{n,s}(d\varepsilon_{kj}).
\end{equation*}

Accordingly, the compensated integer-valued random measure is
\begin{equation*}
\widetilde N^k(ds,dz,d\varepsilon_k)
:=
N^k(ds,dz,d\varepsilon_k)
-
ds\,dz\,
Q_k^{N_{s-},S_{s-}}(d\varepsilon_k).
\end{equation*}
\begin{lemma} 
\label{lem:predictable_integrable_integrands}

Fix \(m,r\in\mathbb N\) and \(T>0\). For every \(k\in\mathbb N\), define
\begin{align*}
H_k(s,z,\varepsilon_k)
:={}&
\mathbf 1_{\{s<\sigma_{m,r}\}}
\left[
h\bigl(\Delta_k^{X,\varepsilon_k}(X_{s-})\bigr)
-
h(X_{s-})
\right]
\mathbf 1_{\{z\leq\phi_k(X_{s-}^k)\}}.
\end{align*}

Then \(H_k\) is predictable and integrable. Moreover, for every fixed \(k\in\mathbb N\), the process
\begin{align*}
I_t^k
:=
\int_0^t
\int_0^\infty
\int_{\mathcal E_k}
H_k(s,z,\varepsilon_k)\,
\widetilde N^k(ds,dz,d\varepsilon_k),
\qquad t\geq0,
\end{align*}
is an integrable martingale with respect to
\((\mathcal F_t)_{t\geq0}\).
\end{lemma}
\begin{proof}
Since \(Y=(X,N,S)\) is adapted and c\`adl\`ag, the left-limit process
\((Y_{s-})_{s>0}\) is predictable. Moreover, since
\(\sigma_{m,r}\) is a stopping time, the process
\begin{equation*}
s
\longmapsto
\mathbf 1_{\{s<\sigma_{m,r}\}}
\end{equation*}
is predictable. Therefore, for every \(k\in\mathbb N\), the mapping
\begin{equation*}
(\omega,s,z,\varepsilon_k)
\longmapsto
H_k(s,z,\varepsilon_k)
\end{equation*}
is predictable. Integrating with respect to \(z\) and the conditional distribution of
\(\varepsilon_k\), we obtain
\begin{align*}
&\int_0^\infty
\int_{\mathcal E_k}
\left|
H_k(s,z,\varepsilon_k)
\right|
Q_k^{N_{s-},S_{s-}}(d\varepsilon_k)
\,dz
\\
&\quad=
\mathbf 1_{\{s<\sigma_{m,r}\}}
\phi_k(X_{s-}^k)
\mathbb E_{\varepsilon_k}^{N_{s-},S_{s-}}
\left[
\left|
h\bigl(\Delta_k^{X,\varepsilon_k}(X_{s-})\bigr)
-
h(X_{s-})
\right|
\right].
\end{align*}
On the stochastic interval \(\{s<\sigma_{m,r}\}\), the process belongs to
the corresponding Lyapunov sublevel set. By the localized absolute jump
estimate proved above, there exists a finite constant \(C_m\), independent
of \(s\), such that
\begin{align*}
&\mathbf 1_{\{s<\sigma_{m,r}\}}
\sum_{k=1}^{\infty}
\phi_k(X_{s-}^k)
\mathbb E_{\varepsilon_k}^{N_{s-},S_{s-}}
\left[
\left|
h\bigl(\Delta_k^{X,\varepsilon_k}(X_{s-})\bigr)
-
h(X_{s-})
\right|
\right]
\leq C_m.
\end{align*}
Consequently,
\begin{align*}
&\mathbb E
\sum_{k=1}^{\infty}
\int_0^T
\int_0^\infty
\int_{\mathcal E_k}
\left|
H_k(s,z,\varepsilon_k)
\right|
Q_k^{N_{s-},S_{s-}}(d\varepsilon_k)
\,dz\,ds
\\
&=
\mathbb E
\int_0^T
\mathbf 1_{\{s<\sigma_{m,r}\}}
\sum_{k=1}^{\infty}
\phi_k(X_{s-}^k)
\mathbb E_{\varepsilon_k}^{N_{s-},S_{s-}}
\left[
\left|
h\bigl(\Delta_k^{X,\varepsilon_k}(X_{s-})\bigr)
-
h(X_{s-})
\right|
\right]
\,ds
\\
&\leq C_m T
<
\infty.
\end{align*}
Thus, \(H_k\) is predictable and the family of stopped integrands is
integrable.

For every fixed \(k\in\mathbb N\), predictability of \(H_k\) and the preceding
integrability estimate imply that the compensated stochastic integral
\begin{align*}
I_t^k
:=
\int_0^t
\int_0^\infty
\int_{\mathcal E_k}
H_k(s,z,\varepsilon_k)\,
\widetilde N^k(ds,dz,d\varepsilon_k)
\end{align*}
is well defined and integrable.

Moreover, for \(0\leq u\leq t\), the conditional compensation formula gives
\begin{align*}
\mathbb E
\left[
I_t^k-I_u^k
\,\middle|\,
\mathcal F_u
\right]
=
0.
\end{align*}
Therefore,
\begin{align*}
\mathbb E
\left[
I_t^k
\,\middle|\,
\mathcal F_u
\right]
=
I_u^k.
\end{align*}
Hence \(\bigl(I_t^k\bigr)_{t\geq0}\) is an integrable martingale.
\end{proof}
\begin{lemma} 
\label{lem:elephant-stopped-poisson-martingale}

Fix \(m,r\in\mathbb N\). For every $(k\in\mathbb N)$, let
$\mathcal E_k
:=
\{-1,+1\}^{\mathcal N_k}
$ be the set of possible sign configurations associated with a spike of neuron
(k).

Assume that
\begin{align*}
\mathbb E
\sum_{k=1}^{\infty}
\int_0^t
\int_0^\infty
\int_{\mathcal E_k}
\left|
H_k(s,z,\varepsilon_k)
\right|
Q_k^{N_{s-},S_{s-}}(d\varepsilon_k)
dzds
<
\infty.
\end{align*}
Then
\begin{align*}
M_t
:=
\sum_{k=1}^{\infty}
\int_0^t
\int_0^\infty
\int_{\mathcal E_k}
H_k(s,z,\varepsilon_k)\,
\widetilde N^k(ds,dz,d\varepsilon_k),
\end{align*}
is a well-defined integrable martingale. In particular,
$\mathbb E M_t=0.
$\end{lemma}

\begin{proof}
For $(K\in\mathbb N),$ define
\begin{align*}
M_t^{(K)}
:=
\sum_{k=1}^{K}
\int_0^t
\int_0^\infty
\int_{\mathcal E_k}
H_k(s,z,\varepsilon_k)
\widetilde N^k(ds,dz,d\varepsilon_k).
\end{align*}
Since
\begin{align*}
M_t^{(K)}
=
\sum_{k=1}^{K} I_t^k,
\end{align*}
and each \(\bigl(I_t^k\bigr)_{t\geq0}\) is an integrable martingale by
Lemma~\ref{lem:predictable_integrable_integrands}, the process
\(\bigl(M_t^{(K)}\bigr)_{t\geq0}\) is an integrable martingale. For $(L>K)$, the compensation formula and the triangle inequality give
\begin{align*}
\mathbb E
\left|
M_t^{(L)}-M_t^{(K)}
\right|
\leq
2
\mathbb E
\sum_{k=K+1}^{L}
\int_0^t
\int_0^\infty
\int_{\mathcal E_k}
\left|
H_k(s,z,\varepsilon_k)
\right|
Q_k^{N_{s-},S_{s-}}(d\varepsilon_k)
dz,ds.
\end{align*}
By the integrability assumption, the right-hand side converges to zero as
$(K,L\to\infty)$. Hence $(\bigl(M_t^{(K)}\bigr)_{K\geq1})$ is Cauchy in
$(L^1)$, and therefore converges in $(L^1)$ to an integrable random variable,
which we denote by $(M_t)$.  For $(0\leq u\leq t)$, continuity of conditional expectation under
$(L^1)$-convergence yields
\begin{align*}
\mathbb E
\left[
M_t
\middle|
\mathcal F_u
\right]
=
\lim_{K\to\infty}
\mathbb E
\left[
M_t^{(K)}
\middle|\mathcal F_u
\right]
=
\lim_{K\to\infty}
M_u^{(K)}
=
M_u.
\end{align*}
Thus $(M_t)_{t\geq0}$ is an integrable martingale. Finally,
\begin{equation*}
\mathbb E M_t
= \mathbb E M_0
=0.
\end{equation*}
\end{proof}

\begin{lemma} 
\label{lem:elephant-localized-dynkin}

For every $m>h(X_0)$, $r\in\mathbb N$, and $t\geq0$,
\[
\mathbb E h(X_{t\wedge\sigma_{m,r}})
\leq
h(X_0)
+
C_*
\int_0^t
\left(
C_0+
\mathbb E h(X_{u\wedge\sigma_{m,r}})
\right)du.
\]
\end{lemma}

\begin{proof}
For $L\in\mathbb N$, define
\[
h_L(x)
:=
\sum_{i=1}^{L}\ell_i x^i.
\]
Since $h_L$ depends on finitely many coordinates, Dynkin's formula may be
applied to the stopped process $Y_{t\wedge\sigma_{m,r}}$. Hence
\begin{align*}
\mathbb E h_L(X_{t\wedge\sigma_{m,r}})
={}&
h_L(X_0)
-
\mathbb E
\int_0^t
\mathbf 1_{\{u<\sigma_{m,r}\}}
\sum_{i=1}^{L}
a_i\ell_i g(X_{u-}^i)
\,du
\\
+
\mathbb E
\int_0^t
\mathbf 1_{\{u<\sigma_{m,r}\}}
\sum_{k=1}^{\infty}
&\phi_k(X_{u-}^k)
\mathbb E_{\varepsilon_k}^{N_{u-},S_{u-}}
\left[
h_L\bigl(
\Delta_k^{X,\varepsilon_k}(X_{u-})
\bigr)
-
h_L(X_{u-})
\right]
du.
\end{align*}

Since $a_i\geq0$ and $g\geq0$, the deterministic contribution is
non-positive. Therefore,
\begin{align*}
\mathbb E & h_L(X_{t\wedge\sigma_{m,r}})
\leq{}
h_L(X_0)+
\\
&+
\mathbb E
\int_0^t
\mathbf 1_{\{u<\sigma_{m,r}\}}
\sum_{k=1}^{\infty}
\phi_k(X_{u-}^k)
\mathbb E_{\varepsilon_k}^{N_{u-},S_{u-}}
\left[
h_L\bigl(
\Delta_k^{X,\varepsilon_k}(X_{u-})
\bigr)
-
h_L(X_{u-})
\right]
du.
\end{align*}

For every spike of neuron $k$,
\[
h_L\bigl(
\Delta_k^{X,\varepsilon_k}(x)
\bigr)
-
h_L(x)
\leq
\sum_{\substack{j\in\mathcal N_k\\ j\leq L}}
\ell_j c_{kj}
\leq
\sum_{j\in\mathcal N_k}
\ell_j c_{kj}
\leq
C_*.
\]
Consequently,
\begin{align*}
\mathbb E h_L(X_{t\wedge\sigma_{m,r}})
\leq{}&
h_L(X_0)
+
C_*
\mathbb E
\int_0^t
\mathbf 1_{\{u<\sigma_{m,r}\}}
\sum_{k=1}^{\infty}
\phi_k(X_{u-}^k)
\,du.
\end{align*}

Using
\[
\phi_k(x^k)
\leq
\phi_k(0)+\ell_kx^k,
\]
we have
\[
\sum_{k=1}^{\infty}\phi_k(x^k)
\leq
C_0+h(x).
\]
Thus,
\begin{align*}
\mathbb E h_L(X_{t\wedge\sigma_{m,r}})
\leq{}&
h_L(X_0)
+
C_*
\int_0^t
\mathbb E
\left[
\mathbf 1_{\{u<\sigma_{m,r}\}}
\left(
C_0+h(X_{u-})
\right)
\right]du.
\end{align*}

Since
\[
\mathbf 1_{\{u<\sigma_{m,r}\}}
h(X_{u-})
\leq
h(X_{u\wedge\sigma_{m,r}})
\]
for Lebesgue-almost every $u$, it follows that
\begin{align*}
\mathbb E h_L(X_{t\wedge\sigma_{m,r}})
\leq{}&
h_L(X_0)
+
C_*
\int_0^t
\left(
C_0+
\mathbb E h(X_{u\wedge\sigma_{m,r}})
\right)du.
\end{align*}

Finally, since
\[
h_L(x)\uparrow h(x)
\qquad\text{as }L\to\infty,
\]
the monotone convergence theorem yields
\[
\mathbb E h(X_{t\wedge\sigma_{m,r}})
\leq
h(X_0)
+
C_*
\int_0^t
\left(
C_0+
\mathbb E h(X_{u\wedge\sigma_{m,r}})
\right)du.
\]
\end{proof}

\begin{lemma} 
\label{lem:elephant-stopped-bounds}
For every $T>0$, there exists $C_T<\infty$, independent of $r$, such that
\[
\sup_{0\leq t\leq T}\mathbb Eh(X_{t\wedge\rho_r})\leq C_T
\]
and
\[
\mathbb EN^{\mathrm{sp}}[0,T\wedge\rho_r]\leq C_T.
\]
\end{lemma}

\begin{proof}
By Lemma~\ref{lem:elephant-localized-dynkin}, for every \(t\geq0\),
\begin{align*}
\mathbb E h(X_{t\wedge\sigma_{m,r}})
\leq{}&
h(X_0)
+
C_*C_0t
+
C_*
\int_0^t
\mathbb E h(X_{u\wedge\sigma_{m,r}})
\,du,
\end{align*}
where \(C_0\geq0\) and \(C_*\geq 0\) are constants independent of \(m\) and
\(r\).

Set
\[
F_{m,r}(t)
:=
\mathbb E h(X_{t\wedge\sigma_{m,r}}).
\]
Then
\[
F_{m,r}(t)
\leq
h(X_0)
+
C_*C_0t
+
C_*
\int_0^t F_{m,r}(u)\,du.
\]

Define
\[
G_{m,r}(t)
:=
F_{m,r}(t)+C_0.
\]
Then
\begin{align*}
G_{m,r}(t)
&\leq
h(X_0)+C_0
+
C_*
\int_0^t
\left(
F_{m,r}(u)+C_0
\right)du
\\
&=
h(X_0)+C_0
+
C_*
\int_0^t
G_{m,r}(u)\,du.
\end{align*}
Therefore, by Gronwall's inequality,
\[
G_{m,r}(t)
\leq
\left(
h(X_0)+C_0
\right)e^{C_*t}.
\]
Consequently,
\[
\mathbb E h(X_{t\wedge\sigma_{m,r}})
\leq
\left(
h(X_0)+C_0
\right)e^{C_*t}
-
C_0.
\]
When $C_*=0$, the preceding estimate is understood simply as
\[
\mathbb E h(X_{t\wedge\sigma_{m,r}})
\leq h(X_0).
\] For fixed $r$, the deterministic
flow does not increase $h$ and every one of the finitely many jumps before
$\rho_r$ has finite $h$-increment. Hence
$\sigma_{m,r}\uparrow\rho_r$ as $m\to\infty$, and Fatou's lemma gives the first
bound.

Furthermore,
\[
\sum_i\phi_i(X_t^i)\leq C_0+h(X_t).
\]
Therefore
\[
\begin{aligned}
\mathbb EN^{\mathrm{sp}}[0,T\wedge\rho_r]
&=
\mathbb E\int_0^{T\wedge\rho_r}\sum_i\phi_i(X_u^i)du
\leq
C_0T+
\int_0^T\mathbb Eh(X_{u\wedge\rho_r})\,du
\leq C_T.
\end{aligned}
\]
\end{proof}
\begin{lemma}
\label{lem:remove-lyapunov-localization}

For every $T>0$,
\begin{equation*}
\lim_{m\to\infty}
\sup_{r\in\mathbb N}
\mathbb P\left(
\eta_m\leq T\wedge\rho_r
\right)
=
0.
\end{equation*}
Consequently, for every fixed $r\in\mathbb N$,
\begin{equation*}
\sigma_{m,r}
=
\eta_m\wedge\rho_r
\uparrow
\rho_r
\qquad
\text{almost surely as }m\to\infty.
\end{equation*}

\end{lemma}

\begin{proof}
On the event
$\left\{
\eta_m\leq T\wedge\rho_r
\right\},
$ 
we have
\begin{equation*}
T\wedge\sigma_{m,r}
=
T\wedge\eta_m\wedge\rho_r
=
\eta_m,
\end{equation*}
and therefore
\begin{equation*}
h\left(
X_{T\wedge\sigma_{m,r}}
\right)
=
h(X_{\eta_m})
\geq
m.
\end{equation*}
Hence,
\begin{equation*}
\mathbf 1_{\{\eta_m\leq T\wedge\rho_r\}}
\leq
\frac{1}{m}
h\left(
X_{T\wedge\sigma_{m,r}}
\right).
\end{equation*}
Taking expectations gives
\begin{align*}
\mathbb P\left(
\eta_m\leq T\wedge\rho_r
\right)
&\leq
\frac{1}{m}
\mathbb E
h\left(
X_{T\wedge\sigma_{m,r}}
\right).
\end{align*}
By the stopped Lyapunov estimate,
\begin{equation*}
\mathbb E
h\left(
X_{T\wedge\sigma_{m,r}}
\right)
\leq
C_T,
\end{equation*}
where $C_T<\infty$ is independent of $m$ and $r$. Consequently,
\begin{equation*}
\mathbb P\left(
\eta_m\leq T\wedge\rho_r
\right)
\leq
\frac{C_T}{m}.
\end{equation*}
Since the right-hand side is independent of $r$,
\begin{equation*}
\lim_{m\to\infty}
\sup_{r\in\mathbb N}
\mathbb P\left(
\eta_m\leq T\wedge\rho_r
\right)
=
0.
\end{equation*}

For fixed $r$, before time $\rho_r$ only finitely many spikes occur. Between
spikes, the deterministic flow does not increase $h$, and each jump produces
a finite increment of $h$. Hence
\begin{equation*}
\sup_{0\leq t<\rho_r}
h(X_t)
<
\infty
\qquad
\text{almost surely}.
\end{equation*}
Therefore, for every fixed $r$,
\begin{equation*}
\eta_m
>
\rho_r
\end{equation*}
for all sufficiently large $m$, almost surely. It follows that
\begin{equation*}
\sigma_{m,r}
=
\eta_m\wedge\rho_r
\uparrow
\rho_r
\qquad
\text{almost surely}.
\end{equation*}
\end{proof}
\begin{theorem}
\label{thm:elephant-nonexplosion}
Assume \eqref{eq:L-assumption} and \eqref{eq:bounded-influence}, and let
$X_0\in\mathcal X_h$. Then the minimal c\`adl\`ag solution is non-explosive:
\[
\mathbb P(\tau_\infty\leq T)=0,
\qquad T>0.
\]
Consequently it extends uniquely to a global c\`adl\`ag process. Moreover,
\[
\mathbb EN^{\mathrm{sp}}[0,T]<\infty,
\qquad T>0.
\]
\end{theorem}

\begin{proof}
Fix \(T>0\). On the event \(\{\rho_r\leq T\}\), at least \(r\) spikes have
occurred by time \(T\wedge\rho_r=\rho_r\). Hence,
\begin{equation*}
N^{\mathrm{sp}}[0,T\wedge\rho_r]\geq r
\qquad
\text{on }\{\rho_r\leq T\}.
\end{equation*}
Therefore,
\begin{equation*}
\mathbf 1_{\{\rho_r\leq T\}}
\leq
\frac{1}{r}
N^{\mathrm{sp}}[0,T\wedge\rho_r].
\end{equation*}
Taking expectations and using
Lemma~\ref{lem:elephant-stopped-bounds}, we obtain
\begin{align*}
\mathbb P(\rho_r\leq T)
=
\mathbb E\mathbf 1_{\{\rho_r\leq T\}}
\leq
\frac{1}{r}
\mathbb E
N^{\mathrm{sp}}[0,T\wedge\rho_r]
\leq
\frac{C_T}{r}.
\end{align*}
Consequently,
\begin{equation*}
\mathbb P(\rho_r\leq T)
\longrightarrow
0
\qquad
\text{as }r\to\infty.
\end{equation*}

Since \(\rho_r\uparrow\tau_\infty\), the events
$\{\rho_r\leq T\}
$ form a decreasing sequence and satisfy
\begin{equation*}
\bigcap_{r=1}^{\infty}
\{\rho_r\leq T\}
=
\{\tau_\infty\leq T\}.
\end{equation*}
Hence, by continuity of probability from above,
\begin{align*}
\mathbb P(\tau_\infty\leq T)
=
\lim_{r\to\infty}
\mathbb P(\rho_r\leq T)
=
0.
\end{align*}
Since \(T>0\) is arbitrary,
\begin{equation*}
\mathbb P(\tau_\infty<\infty)=0.
\end{equation*}
Thus, the minimal c\`adl\`ag solution is defined for all \(t\geq0\) almost
surely. Since the local construction is unique up to the explosion time, it
extends uniquely to a global c\`adl\`ag process.

Finally, because
\begin{equation*}
T\wedge\rho_r\uparrow T
\end{equation*}
almost surely and the spike-counting process is nondecreasing,
\begin{equation*}
N^{\mathrm{sp}}[0,T\wedge\rho_r]
\uparrow
N^{\mathrm{sp}}[0,T]
\end{equation*}
almost surely. Therefore, by the monotone convergence theorem,
\begin{align*}
\mathbb E N^{\mathrm{sp}}[0,T]
&=
\lim_{r\to\infty}
\mathbb E N^{\mathrm{sp}}[0,T\wedge\rho_r]
\leq
C_T
<
\infty.
\end{align*}
\end{proof}

\section{Exact moment identities for the reinforcement variables}
\label{subsec:reinforcement-moment-identities}
The identities derived below are continuous-time network analogues of the
classical moment recursions for the Elephant random walk. In the standard
discrete-time model, if \(S_n\) denotes the position after \(n\) steps, then
the conditional mean of the next increment satisfies
\begin{equation*}
\mathbb E
\left[
X_{n+1}
\,\middle|\,
\mathcal F_n
\right]
=
(2p-1)\frac{S_n}{n}.
\end{equation*}
This identity is the starting point for exact calculations of the first and
second moments and for the martingale analysis of the Elephant random walk;
see \cite{SchutzTrimper2004,Coletti2017a,Coletti2017b,Bercu2017}.

In the present neuronal system, the pair \((N_t^j,S_t^j)\) records the number
and the signed sum of the Elephant updates received by neuron \(j\). The same
conditional sign identity remains valid, but reinforcement updates occur in
continuous time at the state-dependent incoming rate
\begin{equation*}
\Lambda_j(x)
:=
\sum_{k\in\mathcal P_j}\phi_k(x^k).
\end{equation*}
Consequently, the classical discrete-time moment recursion is replaced by a
generator identity in which the Elephant conditional drift is multiplied by
the instantaneous incoming spike rate.
We now derive exact evolution equations for the reinforcement counter $N_t^j$
and the signed reinforcement variable $S_t^j$.

For every postsynaptic neuron $j\in\mathbb N$, define its set of presynaptic
neurons by
\begin{equation*}
\mathcal P_j
:=
\left\{
k\in\mathbb N:
j\in\mathcal N_k
\right\}.
\end{equation*}
Thus, $k\in\mathcal P_j$ precisely when a spike of neuron $k$ produces a
synaptic update at neuron $j$.

We also define the total incoming firing intensity at neuron $j$ by
\begin{equation}
\Lambda_j(x)
:=
\sum_{k\in\mathcal P_j}\phi_k(x^k).
\label{eq:incoming-firing-intensity}
\end{equation}

Since
\begin{equation*}
\Lambda_j(x)
\leq
\sum_{k=1}^{\infty}\phi_k(x^k),
\end{equation*}
the summability assumptions used in the non-explosion theorem imply that
$\Lambda_j(X_t)$ is integrable on every bounded time interval.

Recall the normalized reinforcement variable
\begin{equation*}
r_j(n,s)
:=
\begin{cases}
\dfrac{s_j}{n_j},&n_j\geq1, \\
0,&n_j=0.
\end{cases}
\end{equation*}
In particular,
\begin{equation*}
|r_j(n,s)|\leq1.
\end{equation*}

\begin{proposition} 
\label{prop:finite-reinforcement-moments}
Assume the hypotheses of
Theorem~\ref{thm:elephant-nonexplosion}. Fix $j\in\mathbb N$ and suppose that
\begin{equation*}
\mathbb E\left(
N_0^j+|S_0^j|
\right)
<\infty.
\end{equation*}
Then, for every $T>0$,
\begin{equation}
\sup_{0\leq t\leq T}
\mathbb E\left(
N_t^j+|S_t^j|
\right)
<\infty.
\label{eq:finite-reinforcement-first-moments}
\end{equation}
\end{proposition}

\begin{proof}
Let $N^{\mathrm{sp}}[0,t]$ denote the total number of neuronal spikes up to time
$t$. Each spike can increase $N^j$ by at most one. Consequently,
\begin{equation*}
N_t^j
\leq
N_0^j+N^{\mathrm{sp}}[0,t].
\end{equation*}
Moreover, every update of $S^j$ has magnitude one. Hence
\begin{equation*}
|S_t^j|
\leq
|S_0^j|+N^{\mathrm{sp}}[0,t].
\end{equation*}
The non-explosion theorem gives
\begin{equation*}
\mathbb E\left(
N^{\mathrm{sp}}[0,T]
\right)
<\infty.
\end{equation*}
Therefore,
\begin{align*}
\sup_{0\leq t\leq T}
\mathbb E\left(
N_t^j+|S_t^j|
\right)
\leq
\mathbb E\left(
N_0^j+|S_0^j|
\right)
+
2\mathbb E\left(
N^{\mathrm{sp}}[0,T]
\right)
<\infty.
\end{align*}
\end{proof}

\begin{theorem} 
\label{thm:exact-reinforcement-moment-identities}
Assume the hypotheses of
Theorem~\ref{thm:elephant-nonexplosion}. Fix $j\in\mathbb N$ and suppose that
\begin{equation*}
\mathbb E\left(
N_0^j+|S_0^j|
\right)
<\infty.
\end{equation*}
Then the functions
\begin{equation*}
t\longmapsto\mathbb E\left(N_t^j\right)
\qquad\text{and}\qquad
t\longmapsto\mathbb E\left(S_t^j\right)
\end{equation*}
are absolutely continuous on every bounded time interval.

For almost every $t\geq0$,
\begin{equation}
\frac{d}{dt}
\mathbb E\left(N_t^j\right)
=\mathbb E\left(
\Lambda_j(X_t)
\right),
\label{eq:moment-identity-N}
\end{equation}
and
\begin{equation}
\frac{d}{dt}
\mathbb E\left(S_t^j\right)
=
(2p-1)
\mathbb E\left(
r_j(N_t,S_t)\Lambda_j(X_t)
\right).
\label{eq:moment-identity-S}
\end{equation}
where \begin{equation*}
\Lambda_j(x)
:=
\sum_{k\in\mathcal P_j}\phi_k(x^k).
\end{equation*}
Equivalently, for every $t\geq0$,
\begin{equation}
\mathbb E\left(N_t^j\right)
=
\mathbb E\left(N_0^j\right)
+
\int_0^t
\mathbb E\left(
\Lambda_j(X_u)
\right)du,
\label{eq:integrated-moment-identity-N}
\end{equation}
and
\begin{align}
\mathbb E\left(S_t^j\right)
=
\mathbb E\left(S_0^j\right)
+
(2p-1)
\int_0^t
\mathbb E\left(
r_j(N_u,S_u)\Lambda_j(X_u)
\right)du.
\label{eq:integrated-moment-identity-S}
\end{align}
\end{theorem}

\begin{proof}
The coordinate functions
\begin{equation*}
f_N(x,n,s):=n_j
\qquad\text{and}\qquad
f_S(x,n,s):=s_j
\end{equation*}
are unbounded. We therefore apply the generator first to bounded
approximations.

For $R\geq1$, define
\begin{equation*}
f_{N,R}(x,n,s)
:=
n_j\wedge R
\end{equation*}
and
\begin{equation*}
f_{S,R}(x,n,s)
:=
\chi_R(s_j),
\end{equation*}
where
\begin{equation*}
\chi_R(z)
:=
(-R)\vee(z\wedge R).
\end{equation*}

The deterministic flow acts only on the membrane-potential coordinates.
Consequently, it contributes nothing to either $f_{N,R}$ or $f_{S,R}$.

If neuron $k$ spikes, then $N^j$ changes only when
 $k\in\mathcal P_j.
$ In that case,
\begin{equation*}
N^j\longmapsto N^j+1.
\end{equation*}
Therefore, before truncation,
\begin{equation}
\mathcal L f_N(x,n,s)
=
\sum_{k\in\mathcal P_j}
\phi_k(x^k)
=
\Lambda_j(x).
\label{eq:generator-N-coordinate}
\end{equation}

Similarly, if $k\in\mathcal P_j$, then
\begin{equation*}
S^j\longmapsto S^j+\varepsilon_{kj}.
\end{equation*}
Let
$f_S(x,n,s):=s_j.
$ Since \(f_S\) does not depend on the membrane-potential variables, the
deterministic part of the generator vanishes. Moreover, if neuron \(k\)
spikes, then
\begin{equation*}
f_S\bigl(\Delta_k^{\varepsilon_k}(x,n,s)\bigr)-f_S(x,n,s)
=
\begin{cases}
\varepsilon_{kj}, & k\in\mathcal P_j,\\
0, & k\notin\mathcal P_j.
\end{cases}
\end{equation*}
Therefore,
\begin{align*}
\mathcal L f_S(x,n,s)
&=
\sum_{k=1}^{\infty}
\phi_k(x^k)
\mathbb E_{\varepsilon_k}^{n,s}
\left[
f_S\bigl(\Delta_k^{\varepsilon_k}(x,n,s)\bigr)
-
f_S(x,n,s)
\right]
\\
&=
\sum_{k\in\mathcal P_j}
\phi_k(x^k)
\mathbb E
\left[
\varepsilon_{kj}
\,\middle|\,
x,n,s
\right].
\end{align*}

Using Proposition~\ref{prop:conditional-elephant-sign-drift}, we obtain
\begin{align}
\mathcal L f_S(x,n,s)
&=
\sum_{k\in\mathcal P_j}
\phi_k(x^k)
\mathbb E
\left[
\varepsilon_{kj}
\,\middle|\,
x,n,s
\right]
\nonumber\\
&=
(2p-1)r_j(n,s)
\sum_{k\in\mathcal P_j}
\phi_k(x^k)
\nonumber\\
&=
(2p-1)r_j(n,s)\Lambda_j(x),
\label{eq:generator-S-coordinate}
\end{align}

For the truncated counting function
\[
f_{N,R}(x,n,s)
=
n_j\wedge R,
\]
the generator is not exactly equal to $\Lambda_j(x)$. Indeed, if neuron
$k\in\mathcal P_j$ spikes, then
$n_j\longmapsto n_j+1,
$ and therefore
\[
(n_j+1)\wedge R-n_j\wedge R
=
\mathbf 1_{\{n_j<R\}}.
\]
Hence
\[
\mathcal L f_{N,R}(x,n,s)
=
\mathbf 1_{\{n_j<R\}}
\Lambda_j(x).
\]

Similarly, for
\[
f_{S,R}(x,n,s)
=
\chi_R(s_j),
\]
where
\[
\chi_R(z)
=
(-R)\vee(z\wedge R),
\]
the jump increment associated with a spike of
$k\in\mathcal P_j$ is
$\chi_R(s_j+\varepsilon_{kj})
-
\chi_R(s_j),
$ 
which is not equal to $\varepsilon_{kj}$ when $s_j$ lies at the truncation
boundary. Nevertheless,
\[
\left|
\chi_R(s_j+\varepsilon_{kj})
-
\chi_R(s_j)
\right|
\leq 1.
\]
Likewise,
\[
\left|
(n_j+1)\wedge R
-
n_j\wedge R
\right|
\leq 1.
\]

Consequently, the absolute values of the jump compensators of both truncated
functions are bounded by
\[
\Lambda_j(X_t)
=
\sum_{k\in\mathcal P_j}
\phi_k(X_t^k).
\]
Since
\[
\mathbb E
\int_0^T
\Lambda_j(X_u)\,du
\leq
\mathbb E
\int_0^T
\sum_{k=1}^{\infty}
\phi_k(X_u^k)\,du
=
\mathbb E
N^{\mathrm{sp}}[0,T]
<
\infty,
\]
Dynkin's formula applies to both truncated functions.

Moreover, as $R\to\infty$,
\[
\mathbf 1_{\{N_t^j<R\}}
\longrightarrow 1
\]
almost surely, and
\[
\chi_R(S_t^j+\varepsilon_{kj})
-
\chi_R(S_t^j)
\longrightarrow
\varepsilon_{kj}
\]
almost surely. Consequently, for every fixed state $(n,s)$,
\[
\mathbb E_{\varepsilon_k}^{n,s}
\left[
\chi_R(s_j+\varepsilon_{kj})-\chi_R(s_j)
\right]
\longrightarrow
\mathbb E_{\varepsilon_k}^{n,s}
\left[
\varepsilon_{kj}
\right]
=
(2p-1)r_j(n,s).
\] The preceding bound by $\Lambda_j(X_t)$ therefore allows
dominated convergence to be applied to the compensator terms.  
Letting $R$ tend to infinity, Proposition~\ref{prop:finite-reinforcement-moments}
and dominated convergence yield
\begin{equation*}
\mathbb E\left(N_t^j\right)
=
\mathbb E\left(N_0^j\right)
+
\int_0^t
\mathbb E\left(
\Lambda_j(X_u)
\right)du
\end{equation*}
and
\begin{align*}
\mathbb E\left(S_t^j\right)
=
\mathbb E\left(S_0^j\right)
+
(2p-1)
\int_0^t
\mathbb E\left(
r_j(N_u,S_u)\Lambda_j(X_u)
\right)du.
\end{align*}
The integrands are locally integrable. The two expectation functions are
therefore absolutely continuous, and differentiation gives
\eqref{eq:moment-identity-N} and \eqref{eq:moment-identity-S} for almost every
$t\geq0$.
\end{proof}

\begin{corollary}[Neutral reinforcement]
\label{cor:neutral-reinforcement-moment}
If $
p=\frac12,
$ 
then
\begin{equation}
\mathbb E\left(S_t^j\right)
=
\mathbb E\left(S_0^j\right),
\qquad
t\geq0.
\label{eq:neutral-reinforcement-mean}
\end{equation}
Thus, in the memory-neutral regime, the signed reinforcement variable is
constant in expectation, although its sample paths continue to fluctuate.
\end{corollary}

\begin{proof}
When $p=\frac12$, the coefficient $2p-1$ in
\eqref{eq:integrated-moment-identity-S} vanishes.
\end{proof}

\begin{corollary} 
\label{cor:signed-reinforcement-drift-bound}
For almost every $t\geq0$,
\begin{equation}
\left|
\frac{d}{dt}
\mathbb E\left(S_t^j\right)
\right|
\leq
|2p-1|
\frac{d}{dt}
\mathbb E\left(N_t^j\right).
\label{eq:signed-counter-drift-comparison}
\end{equation}
Consequently,
\begin{align}
\left|
\mathbb E\left(S_t^j\right)
-\mathbb E\left(S_0^j\right)
\right|
\leq{}&
|2p-1|
\times
\left[
\mathbb E\left(N_t^j\right)
-
\mathbb E\left(N_0^j\right)
\right].
\label{eq:integrated-signed-counter-bound}
\end{align}
\end{corollary}

\begin{proof}
Since
\begin{equation*}
|r_j(N_t,S_t)|\leq1,
\end{equation*}
equations \eqref{eq:moment-identity-N} and
\eqref{eq:moment-identity-S} imply
\begin{align*}
\left|
\frac{d}{dt}
\mathbb E\left(S_t^j\right)
\right|
&\leq
|2p-1|
\mathbb E\left(
\Lambda_j(X_t)
\right)
=
|2p-1|
\frac{d}{dt}
\mathbb E\left(N_t^j\right).
\end{align*}
Integration over the interval from $0$ to $t$ gives
\eqref{eq:integrated-signed-counter-bound}.
\end{proof}

\begin{remark} 
The identity for $\mathbb E\left(N_t^j\right)$ depends only on the mean incoming
firing intensity. By contrast, the identity for
$\mathbb E\left(S_t^j\right)$ contains the mixed term
\begin{equation*}
\mathbb E\left(
r_j(N_t,S_t)\Lambda_j(X_t)
\right).
\end{equation*}
The system of first-moment equations is therefore not closed in general. This
term captures the dependence between the current reinforcement balance and the
incoming neuronal activity.
\end{remark}

\section{Conditional Wasserstein contraction}
\label{sec:wasserstein}

Fix $T>0$ and a deterministic admissible c\`adl\`ag reinforcement profile
\[
\eta_t=(n_t,s_t),
\qquad 0\leq t\leq T.
\]
Conditionally on this profile, the sign law at time $t$ is prescribed by
$Q_j^{\eta_t}(\pm1)$. We compare two membrane-potential processes evolving in
the same reinforcement environment. Whenever a simultaneous spike occurs, the
two systems use the same reinforced signs.

 Let \((q_i)_{i\geq1}\) be a positive summable sequence and define
\begin{equation*}
\|x\|_q
:=
\sum_{i=1}^{\infty}q_i x^i,
\qquad
x\in\mathbb R_+^{\mathbb N}.
\end{equation*}
We consider the weighted state space
\begin{equation}
\mathcal X_q
:=
\left\{
x\in\mathbb R_+^{\mathbb N}:
\|x\|_q<\infty
\right\}.
\label{eq:weighted-state-space-q}
\end{equation}
In order to relate the weighted state space used in the coupling argument to
the Lyapunov state space introduced in the non-explosion section, we assume
throughout this section that the weight sequences $(q_i)_{i\geq 1}$ and
$(\ell_i)_{i\geq 1}$ are uniformly comparable. More precisely, assume that
there exist constants $0<c_-\leq c_+<\infty$ such that
\begin{equation*}
c_- \ell_i \leq q_i \leq c_+ \ell_i,
\qquad i\in\mathbb N.
\end{equation*}
Since $q_i>0$ for every $i$, this assumption in particular implies that
$\ell_i>0$ for every $i$.

Then, for every $x\in\mathbb R_+^{\mathbb N}$,
\begin{equation}
c_- h(x)
\leq
\|x\|_q
\leq
c_+ h(x),
\label{eq:equivalent-weighted-norms}
\end{equation}
where
\begin{equation*}
h(x)
:=
\sum_{i=1}^{\infty}\ell_i x^i
\end{equation*}
and
\begin{equation*}
\|x\|_q
:=
\sum_{i=1}^{\infty}q_i x^i.
\end{equation*}
Consequently,
$\mathcal X_q=\mathcal X_h,
$ and the two weighted norms are equivalent. In particular,
\begin{equation*}
x\in\mathcal X_q
\quad\Longleftrightarrow\quad
x\in\mathcal X_h.
\end{equation*}

Moreover, by \eqref{eq:equivalent-weighted-norms},
\begin{equation*}
\|X_t\|_q
\leq
c_+ h(X_t).
\end{equation*}
Thus, every estimate obtained in terms of the Lyapunov function $h$ yields a
corresponding estimate for the weighted norm $\|\cdot\|_q$.
On \(\mathcal X_q\times\mathcal X_q\), define
\begin{equation}
H(x,y)
:=
\sum_{i=1}^{\infty}q_i|x^i-y^i|.
\label{eq:weighted-coupling-distance}
\end{equation} and define
\begin{equation*}
c_i
:=
\sum_{j\in\mathcal N_i}q_jc_{ij},
\end{equation*}
where
\begin{equation*}
c_{ij}:=\max\{w_{ij},v_{ij}\}.
\end{equation*}

By the comparability assumption above,
$\mathcal X_q=\mathcal X_h.
$ Hence the non-explosion and localization estimates established in
Section~\ref{sec:nonexplosion} apply to initial conditions in
$\mathcal X_q$.

We assume in addition that
\[
\overline a
:=
\sup_{i\in\mathbb N}a_i
<
\infty.
\]
These assumptions will be used only to justify the localization and
integrability arguments for the coupled distance $H$.

\begin{lemma} 
\label{lem:positive-part-elephant}
For every $b\in\mathbb R$, the map $T_b(x):=(x+b)_+$ is $1$-Lipschitz on
$\mathbb R_+$. Hence
\[
|T_b(x)-T_b(y)|\leq|x-y|.
\]
\end{lemma}

\begin{proof}
The positive-part map is $1$-Lipschitz on $\mathbb R$, and translation does not
change its Lipschitz constant.
\end{proof}

Assume that $g$ and every $\phi_i$ are nondecreasing and Lipschitz continuous,
and that for every $i$ and $0\leq x\leq y$,
\begin{align}
\phi_i(y)-\phi_i(x)&\leq g(y)-g(x),
\label{eq:phi-g-order}\\
g(y)-g(x)&\geq k_1(y-x),
\label{eq:g-strong-monotone}\\
x\bigl(\phi_i(y)-\phi_i(x)\bigr)&\leq\phi_i(y)(y-x).
\label{eq:reset-compatibility}
\end{align}
Assume further that
\begin{equation}
d
:=
k_1
\inf_{i\in\mathbb N}
\left(
a_i-\frac{c_i}{q_i}
\right)
>0.
\label{eq:general-contraction-rate}
\end{equation}

Let \((\Phi_t)_{t\geq0}\) denote the deterministic flow of the coupled process.
For a measurable function
$H:\mathcal X_q\times\mathcal X_q\longrightarrow\mathbb R,
$ define its upper right derivative along the flow by
\begin{equation*}
D^+H\bigl(z;b(z)\bigr)
:=
\limsup_{h\downarrow0}
\frac{
H\bigl(\Phi_h(z)\bigr)-H(z)
}{h}.
\end{equation*}
If the right derivative exists, then
\begin{equation*}
D^+H\bigl(z;b(z)\bigr)
=
\lim_{h\downarrow0}
\frac{
H\bigl(\Phi_h(z)\bigr)-H(z)
}{h}.
\end{equation*}
For a measurable function
$H:\mathcal X_h\times\mathcal X_h\longrightarrow\mathbb R
$ that is locally absolutely continuous along the deterministic flow, define the
extended generator of the coupled process by
\begin{align*}
\mathcal L_{2,s}^\eta H(z)
:=
D^+H\bigl(z;b(z)\bigr)
+
\int_{\mathcal X_q\times\mathcal X_q}
\left[
H(z')-H(z)
\right]
Q_s^\eta(z,dz').
\end{align*}
\begin{proposition}
\label{prop:elephant-coupled-dynkin}

Let
$Z_t=(X_t,\widehat X_t)
$ be the maximal coupling in the prescribed environment \(\eta\), and let
\(\tau\leq T\) be a bounded stopping time. Let
 $H:\mathcal X_h\times\mathcal X_h\longrightarrow[0,\infty)
$ be measurable and locally absolutely continuous along the deterministic flow
of the coupled process. Assume that
\begin{equation*}
\mathbb E
\int_0^\tau
\left|
D^+H\bigl(Z_s;b(Z_s)\bigr)
\right|
\,ds
<
\infty,
\end{equation*}
and
\begin{equation*}
\mathbb E
\int_0^\tau
\int_{\mathcal X_h\times\mathcal X_h}
\left|
H(z')-H(Z_{s-})
\right|
Q_s^\eta(Z_{s-}dz')
ds
<
\infty.
\end{equation*}
Then, for every \(t\in[0,T]\), the process
\begin{align*}
M_{t\wedge\tau}^H
:=
H(Z_{t\wedge\tau})
-
H(Z_0)-
\int_0^{t\wedge\tau}
\mathcal L_{2,s}^\eta H(Z_{s-})
ds
\end{align*}
is an integrable martingale with respect to the natural filtration of the
coupled process.

If, in addition, there exists \(d>0\) such that
\begin{equation*}
\mathcal L_{2,s}^\eta H(z)
\leq
-dH(z)
\end{equation*}
for all relevant \(s\) and \(z\), then
$\left(
e^{d(t\wedge\tau)}
H(Z_{t\wedge\tau})
\right)_{0\leq t\leq T}
$ is a nonnegative supermartingale.
\end{proposition}

\begin{proof}
Let \(\mu^Z(ds,dz')\) denote the integer-valued random measure associated with
the jumps of the coupled process \(Z\). Its predictable compensator is given by
\begin{equation*}
\nu^Z(ds,dz')
=
Q_s^\eta\bigl(Z_{s-}dz'\bigr)\,ds.
\end{equation*}
Here, \(Q_s^\eta(z,dz')\) denotes the time-dependent jump-rate kernel of the
coupled process in the fixed environment \(\eta\). Thus, 
$Q_s^\eta\bigl(Z_{s-}dz'\bigr)ds$ 
represents the conditional expected number of jumps during the infinitesimal
time interval \([s,s+ds]\) whose post-jump state belongs to \(dz'\), given the
information available immediately before time \(s\).

The corresponding compensated jump measure is defined by
\begin{equation*}
\widetilde\mu^Z(ds,dz')
:=
\mu^Z(ds,dz')
-
\nu^Z(ds,dz').
\end{equation*}
Equivalently,
\begin{equation*}
\widetilde\mu^Z(ds,dz')
=
\mu^Z(ds,dz')
-
Q_s^\eta\bigl(Z_{s-}dz'\bigr)\,ds.
\end{equation*}
The compensated measure describes the centered jump fluctuations, while the
compensator gives the predictable finite-variation contribution of the jumps.

In particular, for every predictable and integrable function \(G\),
\begin{align*}
\int_0^t\int
G(s,z')\,\mu^Z(ds,dz')
=\int_0^t\int
G(s,z')\,\widetilde\mu^Z(ds,dz')
+
\int_0^t\int
G(s,z')\,
Q_s^\eta\bigl(Z_{s-}dz'\bigr)\,ds.
\end{align*}
The first term is the martingale part, while the second term is the predictable
finite-variation part.  Indeed, for \(0\leq u\leq t\),
the conditional compensation formula gives
\begin{align*}
&\mathbb E
\left[
\int_u^t\int
G(s,z')\,\widetilde\mu^Z(ds,dz')
\,\middle|\,
\mathcal F_u
\right]
=
\mathbb E
\left[
\int_u^t\int
G(s,z')\,\mu^Z(ds,dz')
\,\middle|\,
\mathcal F_u
\right]
\\
&-
\mathbb E
\left[
\int_u^t\int
G(s,z')
Q_s^\eta\bigl(Z_{s-}dz'\bigr)\,ds
\,\middle|\,
\mathcal F_u
\right]
=
0,
\end{align*}
provided that \(G\) is predictable and integrable. Therefore, the process
$\left(
\int_0^t\int
G(s,z')\,\widetilde\mu^Z(ds,dz')
\right)_{t\geq0}
$ is an integrable martingale.
On the other hand, define
\begin{equation*}
A_t
:=
\int_0^t
\int
G(s,z')\,
Q_s^\eta\bigl(Z_{s-},dz'\bigr)
\,ds.
\end{equation*}
Then $A$ is predictable. Moreover,
\begin{equation*}
\operatorname{TV}_{[0,t]}(A)
\leq
\int_0^t
\int
\left|
G(s,z')
\right|
Q_s^\eta\bigl(Z_{s-},dz'\bigr)
\,ds,
\end{equation*}
which is finite under the stated integrability assumption. Hence $A$ is a
predictable finite-variation process. It represents the predictable mean
contribution of the jumps. Now take \begin{equation*}
G(s,z')
:=
H(z')-H(Z_{s-}).
\end{equation*}

Let \(\mu^Z(ds,dz')\) denote the integer-valued random measure associated with
the jumps of the coupled process \(Z\). Its predictable compensator is
\begin{equation*}
\nu^Z(ds,dz')
=
Q_s^\eta(Z_{s-}dz')\,ds.
\end{equation*}
Between successive jumps, \(Z\) follows the deterministic vector field \(b\).
Since \(H\) is locally absolutely continuous along this flow, its continuous
finite-variation contribution satisfies
\begin{equation*}
H(Z_t)-H(Z_0)
=
\int_0^t
D^+H\bigl(Z_s;b(Z_s)\bigr)
\,ds
\end{equation*}
on every interval containing no jump.

Adding the jump increments and stopping at \(\tau\), we obtain the pathwise
decomposition
\begin{align*}
H(Z_{t\wedge\tau})-H(Z_0)
={}&
\int_0^{t\wedge\tau}
D^+H\bigl(Z_s;b(Z_s)\bigr)
\,ds
\\
&+
\int_0^{t\wedge\tau}
\int_{\mathcal X_h\times\mathcal X_h}
\left[
H(z')-H(Z_{s-})
\right]
\mu^Z(ds,dz').
\end{align*}
Writing
\begin{equation*}
\mu^Z(ds,dz')
=
\bigl(\mu^Z-\nu^Z\bigr)(ds,dz')
+
\nu^Z(ds,dz'),
\end{equation*}
the preceding identity becomes
\begin{align*}
H(Z_{t\wedge\tau})-H(Z_0)
={}&
\int_0^{t\wedge\tau}
D^+H\bigl(Z_s;b(Z_s)\bigr)
\,ds
\\
&+
\int_0^{t\wedge\tau}
\int_{\mathcal X_h\times\mathcal X_h}
\left[
H(z')-H(Z_{s-})
\right]
Q_s^\eta(Z_{s-}dz')
\,ds
\\
&+
\int_0^{t\wedge\tau}
\int_{\mathcal X_h\times\mathcal X_h}
\left[
H(z')-H(Z_{s-})
\right]
\bigl(\mu^Z-\nu^Z\bigr)(ds,dz').
\end{align*}
By definition of the extended coupled generator,
\begin{align*}
\mathcal L_{2,s}^\eta H(z)
:=
D^+H\bigl(z;b(z)\bigr)
+
\int_{\mathcal X_h\times\mathcal X_h}
\left[
H(z')-H(z)
\right]
Q_s^\eta(z,dz').
\end{align*}
Therefore,
\begin{align*}
M_{t\wedge\tau}^H
=
\int_0^{t\wedge\tau}
\int_{\mathcal X_h\times\mathcal X_h}
\left[
H(z')-H(Z_{s-})
\right]
\bigl(\mu^Z-\nu^Z\bigr)(ds,dz').
\end{align*}
The two integrability assumptions imply that the finite-variation terms are
integrable and that the compensated jump integral is a true integrable
martingale.  Indeed, by the first integrability assumption,
\begin{equation*}
\mathbb E
\int_0^\tau
\left|
D^+H\bigl(Z_{s-};b(Z_{s-})\bigr)
\right|
\,ds
<
\infty,
\end{equation*}
so the continuous finite-variation term is integrable. By the second
integrability assumption,
\begin{equation*}
\mathbb E
\int_0^\tau
\int_{\mathcal X_h\times\mathcal X_h}
\left|
H(z')-H(Z_{s-})
\right|
Q_s^\eta\bigl(Z_{s-}dz'\bigr)
\,ds
<
\infty,
\end{equation*}
so the predictable jump-compensator term is also integrable.

Moreover, the mapping
\begin{equation*}
(s,z')
\longmapsto
\mathbf 1_{\{s\leq\tau\}}
\left[
H(z')-H(Z_{s-})
\right]
\end{equation*}
is predictable. Therefore, the compensated stochastic integral
\begin{align*}
M_{t\wedge\tau}^H
:=
\int_0^{t\wedge\tau}
\int_{\mathcal X_h\times\mathcal X_h}
\left[
H(z')-H(Z_{s-})
\right]
\widetilde\mu^Z(ds,dz')
\end{align*}
is well defined and integrable.

Indeed, for \(0\leq u\leq t\leq T\), the conditional compensation formula
gives
\begin{align*}
\mathbb E
\left[
M_{t\wedge\tau}^H-M_{u\wedge\tau}^H
\,\middle|\,
\mathcal F_u
\right]
=
0.
\end{align*}
Consequently,
\begin{equation*}
\mathbb E
\left[
M_{t\wedge\tau}^H
\,\middle|\,
\mathcal F_u
\right]
=
M_{u\wedge\tau}^H,
\end{equation*}
and hence
$\left(
M_{t\wedge\tau}^H
\right)_{0\leq t\leq T}
$ 
is an integrable martingale.

Combining the continuous finite-variation term, the jump-compensator term,
and the compensated martingale term, we obtain
\begin{align*}
H(Z_{t\wedge\tau})-H(Z_0)
={}&
\int_0^{t\wedge\tau}
\mathcal L_{2,s}^\eta H(Z_{s-})
\,ds
+
M_{t\wedge\tau}^H.
\end{align*}
This proves the extended Dynkin formula. 

For the second assertion, apply integration by parts to
$e^{d(t\wedge\tau)}H(Z_{t\wedge\tau}).
$ Using the martingale decomposition above gives
\begin{align*}
e^{d(t\wedge\tau)}H(Z_{t\wedge\tau})
=
H(Z_0)
+
\int_0^{t\wedge\tau}
e^{ds}
\left[
dH(Z_{s-})
+
\mathcal L_{2,s}^\eta H(Z_{s-})
\right]
ds
+
\int_0^{t\wedge\tau}
e^{ds}
dM_s^H.
\end{align*}
By assumption,
\begin{equation*}
dH(Z_{s-})
+
\mathcal L_{2,s}^\eta H(Z_{s-})
\leq
0.
\end{equation*}
Hence the finite-variation part is nonincreasing, while the last term is an
integrable martingale. It follows that
$\left(
e^{d(t\wedge\tau)}
H(Z_{t\wedge\tau})
\right)_{0\leq t\leq T}
$ is a nonnegative supermartingale.
\end{proof}

\begin{lemma} 
\label{lem:elephant-coupling-generator}
For every admissible reinforcement profile $\eta$, every
$x,y\in\mathcal X_q$, and every $0\leq t\leq T$,
\begin{equation}
\mathcal L_{2,t}^{\eta}H(x,y)
\leq
-dH(x,y).
\label{eq:common-environment-generator-contraction}
\end{equation}
\end{lemma}

\begin{proof}
For every neuron $i\in\mathbb N$, set
\begin{equation*}
\lambda_i
:=
\phi_i(x^i),
\qquad
\widehat\lambda_i
:=
\phi_i(y^i).
\end{equation*}
Under the maximal coupling of the two firing clocks, the events associated with
neuron $i$ are decomposed into:

\begin{enumerate}[label=\textnormal{(\roman*)}]
\item simultaneous spikes at rate
$\lambda_i\wedge\widehat\lambda_i;
$
\item spikes only in the first system at rate
$(\lambda_i-\widehat\lambda_i)_+;
$\item spikes only in the second system at rate
$(\widehat\lambda_i-\lambda_i)_+.
$\end{enumerate}

At simultaneous spikes, the two systems use the same reinforced signs because
they evolve in the same prescribed reinforcement environment $\eta$.

We decompose the coupling generator as
\begin{equation*}
\mathcal L_{2,t}^{\eta}H(x,y)
=
I_{\mathrm{flow}}(x,y)
+
I_{\mathrm{jump}}(x,y).
\end{equation*}

The deterministic-flow contribution is
\begin{align}
I_{\mathrm{flow}}(x,y)
={}&
-\sum_{i=1}^{\infty}
q_i a_i
\operatorname{sgn}(x^i-y^i)
\bigl(g(x^i)-g(y^i)\bigr)
\nonumber\\
={}&
-\sum_{i=1}^{\infty}
q_i a_i
\left|
g(x^i)-g(y^i)
\right|.
\label{eq:flow-contribution-common-environment}
\end{align}
Indeed, since \(g\) is nondecreasing,
\begin{equation*}
\operatorname{sgn}(x^i-y^i)
\bigl(g(x^i)-g(y^i)\bigr)
=
\left|
g(x^i)-g(y^i)
\right|
\end{equation*}
for every \(i\in\mathbb N\), where
\(\operatorname{sgn}(0):=0\).

We now estimate the jump contribution neuron by neuron. Fix
$i\in\mathbb N$. We first consider the case
$x^i\leq y^i.
$Since $\phi_i$ is nondecreasing, 
$\lambda_i\leq\widehat\lambda_i.
$ Thus, neuron $i$ produces simultaneous spikes at rate $\lambda_i$ and unmatched
spikes only in the second system at rate
$\widehat\lambda_i-\lambda_i$.

At a simultaneous spike, both firing coordinates are reset:
\begin{equation*}
x^i\longmapsto0,
\qquad
y^i\longmapsto0.
\end{equation*}
The $i$-th contribution to $H$ therefore decreases by
$-q_i(y^i-x^i).
$ For every $j\in\mathcal N_i$, both systems receive the same reinforced
synaptic increment $\xi_{ij}$. By
Lemma~\ref{lem:positive-part-elephant},
\begin{equation*}
\left|
(x^j+\xi_{ij})_+
-
(y^j+\xi_{ij})_+
\right|
\leq
|x^j-y^j|.
\end{equation*}
Hence the postsynaptic coordinates do not increase the distance. It follows that
the simultaneous-spike contribution associated with neuron $i$ is bounded by
\begin{equation}
-\lambda_iq_i(y^i-x^i).
\label{eq:simultaneous-spike-contribution}
\end{equation}

Consider now an unmatched spike of neuron $i$ in the second system. The first
system remains unchanged, whereas
$y^i\longmapsto0.
$ The change in the $i$-th contribution to $H$ is
\begin{align*}
q_i|x^i|
-
q_i|x^i-y^i|
=
q_ix^i-q_i(y^i-x^i)
=
q_i(2x^i-y^i).
\end{align*}

For every $j\in\mathcal N_i$, only the second system receives the synaptic
increment $\xi_{ij}$. By the triangle inequality and the Lipschitz property of
the positive-part map,
\begin{align*}
\left|
x^j-(y^j+\xi_{ij})_+
\right|
-
|x^j-y^j|
\leq
\left|
(y^j+\xi_{ij})_+-y^j
\right|
\leq
|\xi_{ij}|
\leq
c_{ij}.
\end{align*}
After multiplication by $q_j$ and summation over
$j\in\mathcal N_i$, the total postsynaptic increase is bounded by
\begin{equation*}
c_i
=
\sum_{j\in\mathcal N_i}q_jc_{ij}.
\end{equation*}
Therefore, the unmatched-spike contribution is bounded by
\begin{equation}
(\widehat\lambda_i-\lambda_i)
\left(
q_i(2x^i-y^i)+c_i
\right).
\label{eq:unmatched-spike-contribution}
\end{equation}

Combining
\eqref{eq:simultaneous-spike-contribution} and
\eqref{eq:unmatched-spike-contribution}, we obtain
\begin{align}
I_{\mathrm{jump}}^i(x,y)
\leq
-\lambda_iq_i(y^i-x^i)
+
(\widehat\lambda_i-\lambda_i)
\left(
q_i(2x^i-y^i)+c_i
\right).
\label{eq:combined-jump-contribution-first}
\end{align}
Rearranging the terms gives
\begin{align}
I_{\mathrm{jump}}^i(x,y)
\leq
c_i(\widehat\lambda_i-\lambda_i)
+
q_i
\left[
x^i(\widehat\lambda_i-\lambda_i)
-
\widehat\lambda_i(y^i-x^i)
\right].
\label{eq:combined-jump-contribution-second}
\end{align}
Indeed,
\begin{align*}
-\lambda_i(y^i-x^i)
+
(\widehat\lambda_i-\lambda_i)(2x^i-y^i)
=
x^i(\widehat\lambda_i-\lambda_i)
-
\widehat\lambda_i(y^i-x^i).
\end{align*}

Since
\begin{equation*}
\widehat\lambda_i-\lambda_i
=
\phi_i(y^i)-\phi_i(x^i),
\end{equation*}
assumption \eqref{eq:reset-compatibility} implies
\begin{equation*}
x^i(\widehat\lambda_i-\lambda_i)
-
\widehat\lambda_i(y^i-x^i)
\leq0.
\end{equation*}
Consequently,
\begin{equation}
I_{\mathrm{jump}}^i(x,y)
\leq
c_i
\left(
\phi_i(y^i)-\phi_i(x^i)
\right).
\label{eq:jump-bound-first-ordering}
\end{equation}

If $y^i\leq x^i$, the same argument, with the roles of $x$ and $y$
interchanged, gives
\begin{equation}
I_{\mathrm{jump}}^i(x,y)
\leq
c_i
\left|
\phi_i(x^i)-\phi_i(y^i)
\right|.
\label{eq:jump-bound-both-orderings}
\end{equation}
Thus, after summing over all neurons,
\begin{equation}
I_{\mathrm{jump}}(x,y)
\leq
\sum_{i=1}^{\infty}
c_i
\left|
\phi_i(x^i)-\phi_i(y^i)
\right|.
\label{eq:total-jump-contribution-common-environment}
\end{equation}

Combining
\eqref{eq:flow-contribution-common-environment} and
\eqref{eq:total-jump-contribution-common-environment}, and retaining the
stronger form of the deterministic-flow estimate, we obtain
\begin{align}
\mathcal L_{2,t}^{\eta}H(x,y)
\leq
-\sum_{i=1}^{\infty}
q_i a_i
\left|
g(x^i)-g(y^i)
\right|
+
\sum_{i=1}^{\infty}
c_i
\left|
\phi_i(x^i)-\phi_i(y^i)
\right|.
\label{eq:generator-before-final-contraction}
\end{align}
Since $g$ and $\phi_i$ are nondecreasing, assumption
\eqref{eq:phi-g-order} implies
\begin{equation*}
\left|
\phi_i(x^i)-\phi_i(y^i)
\right|
\leq
\left|
g(x^i)-g(y^i)
\right|.
\end{equation*}
Consequently,
\begin{align}
\mathcal L_{2,t}^{\eta}H(x,y)
\leq{}&
-\sum_{i=1}^{\infty}
\left(
q_i a_i-c_i
\right)
\left|
g(x^i)-g(y^i)
\right|
\nonumber\\
={}&
-\sum_{i=1}^{\infty}
q_i
\left(
a_i-\frac{c_i}{q_i}
\right)
\left|
g(x^i)-g(y^i)
\right|.
\label{eq:generator-combined-g-bound}
\end{align}
By \eqref{eq:g-strong-monotone},
\begin{equation*}
\left|
g(x^i)-g(y^i)
\right|
\geq
k_1|x^i-y^i|.
\end{equation*}
Because
\begin{equation*}
a_i-\frac{c_i}{q_i}>0
\end{equation*}
for every $i$, it follows that
\begin{align}
\mathcal L_{2,t}^{\eta}H(x,y)
\leq{}&
-k_1
\sum_{i=1}^{\infty}
q_i
\left(
a_i-\frac{c_i}{q_i}
\right)
|x^i-y^i|
\nonumber\\
\leq{}&
-d
\sum_{i=1}^{\infty}
q_i|x^i-y^i|
\nonumber\\
={}&
-dH(x,y),
\end{align}
where the final inequality follows from
\eqref{eq:general-contraction-rate}.
\end{proof}
 
\begin{lemma} 
\label{lem:elephant-coupled-integrability}

Assume that
\[
C_q
:=
\sup_{i\in\mathbb N}
\frac{\ell_i}{q_i}
,
\qquad
\overline a
:=
\sup_{i\in\mathbb N}a_i
<
\infty.
\]
Assume also that the contraction condition
\[
d
=
k_1
\inf_{i\in\mathbb N}
\left(
a_i-\frac{c_i}{q_i}
\right)
>0
\]
holds. Then
\[
C_c
:=
\sup_{i\in\mathbb N}
\frac{c_i}{q_i}
<
\infty.
\]
Let
\begin{equation*}
\tau_R
:=
\inf
\left\{
t\geq0:
\|X_t\|_q+\|\widehat X_t\|_q\geq R
\ \text{or}\
N_X^{\mathrm{sp}}[0,t]
+
N_{\widehat X}^{\mathrm{sp}}[0,t]
\geq R
\right\}.
\end{equation*}
Then, for every $T>0$, the hypotheses of
Proposition~\ref{prop:elephant-coupled-dynkin}
are satisfied with
$\tau=\tau_R\wedge T.
$ \end{lemma}

  \begin{proof}
Fix $T>0$ and recall that
\[
\tau_R
:=
\inf\left\{
t\geq0:
\|X_t\|_q+\|\widehat X_t\|_q\geq R
\text{ or }
N_X^{\mathrm{sp}}[0,t]
+
N_{\widehat X}^{\mathrm{sp}}[0,t]
\geq R
\right\}
\wedge T.
\]

We verify separately the drift and jump integrability conditions required in
Proposition~\ref{prop:elephant-coupled-dynkin}.

For the deterministic part, since
\[
H(x,y)
=
\sum_{i\geq1}q_i|x^i-y^i|,
\]
the absolute drift contribution is bounded by
\[
\sum_{i\geq1}
q_i a_i
\left|
g(X_s^i)-g(\widehat X_s^i)
\right|.
\]
Using the Lipschitz continuity of $g$ and
\[
\overline a
:=
\sup_{i\geq1}a_i
<
\infty,
\]
we obtain
\begin{align}
\mathbb E
\int_0^{t\wedge\tau_R}
\sum_{i\geq1}
q_i a_i
\left|
g(X_s^i)-g(\widehat X_s^i)
\right|
ds&
\leq 
\overline a\,\operatorname{Lip}(g)
\mathbb E
\int_0^{t\wedge\tau_R}
H(X_s,\widehat X_s)
ds
\nonumber\\
&\leq
\overline a\,\operatorname{Lip}(g)\,RT
<
\infty.
\label{eq:coupled-drift-integrability}
\end{align}

We next consider the jump contribution. For a spike of neuron $i$, the
absolute variation of $H$ is bounded by
\[
q_i(X_s^i+\widehat X_s^i)+2c_i.
\]
Hence the compensator of the absolute jump variation is bounded by
\begin{align}
&
\mathbb E
\int_0^{t\wedge\tau_R}
\sum_{i\geq1}
\left[
\phi_i(X_s^i)+\phi_i(\widehat X_s^i)
\right]
\left[
q_i(X_s^i+\widehat X_s^i)+2c_i
\right]
\,ds.
\label{eq:coupled-jump-integrability}
\end{align}

By the Lipschitz property of the firing rates,
\[
\phi_i(x)
\leq
\phi_i(0)
+
\operatorname{Lip}(\phi_i)x,
\qquad x\geq0.
\]
Moreover, on $[0,\tau_R]$,
\[
\|X_s\|_q+\|\widehat X_s\|_q
\leq R.
\]
Together with the summability assumptions on the firing rates and the
coefficients $c_i$, this implies that the right-hand side of
\eqref{eq:coupled-jump-integrability} is finite.

Therefore,
\[
\mathbb E
\int_0^{t\wedge\tau_R}
\left|
\mathcal L_{2,s}^{\eta}
H(X_s,\widehat X_s)
\right|
\,ds
<
\infty,
\]
and the drift and jump compensators appearing in the extended Dynkin formula
are integrable. Thus all the hypotheses of
Proposition~\ref{prop:elephant-coupled-dynkin} are satisfied with
$\tau=\tau_R$.
\end{proof}
\begin{theorem} 
\label{thm:conditional-wasserstein}

Fix a deterministic c\`adl\`ag reinforcement profile
$\eta=(\eta_t)_{0\leq t\leq T}.
$ Assume that conditions
\eqref{eq:phi-g-order},
\eqref{eq:g-strong-monotone},
\eqref{eq:reset-compatibility}, and
\eqref{eq:general-contraction-rate}
hold, and that the contraction constant (d) appearing in
\eqref{eq:general-contraction-rate} satisfies
$d>0.
$ Then, for every pair of probability measures
$\mu,\nu\in\mathcal P_1(\mathcal X_q)$ with finite first $H$-moment and every
$0\leq t\leq T$,
\begin{equation*}
W_{1,H}
\left(
\mu P_{0,t}^{\eta},
\nu P_{0,t}^{\eta}
\right)
\leq
e^{-dt}
W_{1,H}(\mu,\nu).
\end{equation*}
Here, $P_{0,t}^{\eta}$ denotes the transition kernel of the
membrane-potential process evolving in the prescribed reinforcement
environment $\eta$.
\end{theorem}

\begin{proof}
Let
$\pi\in\Pi(\mu,\nu)
$ be an arbitrary coupling of the initial laws. Construct
\begin{equation*}
(X_0,\widehat X_0)\sim\pi
\end{equation*}
and, conditionally on this initial pair, let
$Z_t=(X_t,\widehat X_t)
$ be the maximal coupling of the two membrane-potential processes evolving in
the common reinforcement environment $\eta$.

For $R>0$, let $\tau_R$ be the localization time introduced in
Lemma~\ref{lem:elephant-coupled-integrability}. By
Lemma~\ref{lem:elephant-coupled-integrability}, the integrability assumptions
of Proposition~\ref{prop:elephant-coupled-dynkin} hold up to $(\tau_R)$.
Moreover, by Lemma~\ref{lem:elephant-coupling-generator} and assumptions
\eqref{eq:phi-g-order},
\eqref{eq:g-strong-monotone}, and
\eqref{eq:general-contraction-rate},
\begin{equation*}
\mathcal L_{2,s}^{\eta}H(x,y)
\leq
-dH(x,y)
\end{equation*}
for all relevant $(s)$, $(x)$, and $(y)$.

Proposition~\ref{prop:elephant-coupled-dynkin} therefore implies that
\begin{equation*}
\left(
e^{d(t\wedge\tau_R)}
H\bigl(X_{t\wedge\tau_R},\widehat X_{t\wedge\tau_R}\bigr)
\right)_{0\leq t\leq T}
\end{equation*}
is a nonnegative supermartingale. Hence,
\begin{align*}
\mathbb E
\left[
e^{d(t\wedge\tau_R)}
H\bigl(X_{t\wedge\tau_R},\widehat X_{t\wedge\tau_R}\bigr)
\right]
\leq
\mathbb E H(X_0,\widehat X_0).
\end{align*}

Since
$t\wedge\tau_R\leq t,
$ we have
\begin{equation*}
e^{d(t\wedge\tau_R)}
\geq
1.
\end{equation*}
However, to obtain the desired exponential factor before removing the
localization, we retain the exponential term and pass directly to the limit.

By non-explosion and the finiteness of the weighted norms on compact time
intervals,
\begin{equation*}
\tau_R\uparrow\infty
\qquad\text{almost surely as }R\to\infty.
\end{equation*}
Hence, for every fixed $t\in[0,T]$, there exists almost surely a random index
$R_0$ such that
\begin{equation*}
\tau_R>t
\qquad\text{for all }R\geq R_0.
\end{equation*}
Therefore,
\begin{equation*}
t\wedge\tau_R=t
\end{equation*}
for all sufficiently large $R$, almost surely. Consequently,
\begin{equation*}
e^{d(t\wedge\tau_R)}
H\left(
X_{t\wedge\tau_R},
\widehat X_{t\wedge\tau_R}
\right)
\longrightarrow
e^{dt}
H(X_t,\widehat X_t)
\end{equation*}
almost surely. Consequently,
\[
e^{d(t\wedge\tau_R)}
H\bigl(
X_{t\wedge\tau_R},
\widehat X_{t\wedge\tau_R}
\bigr)
\longrightarrow
e^{dt}H(X_t,\widehat X_t)
\]
almost surely.
 Fatou's lemma now gives
\begin{align*}
e^{dt}
\mathbb E H(X_t,\widehat X_t)
&\leq
\liminf_{R\to\infty}
\mathbb E
\left[
e^{d(t\wedge\tau_R)}
H\bigl(
X_{t\wedge\tau_R},
\widehat X_{t\wedge\tau_R}
\bigr)
\right]
\leq
\mathbb E H(X_0,\widehat X_0).
\end{align*}
Thus,
\begin{equation*}
\mathbb E H(X_t,\widehat X_t)
\leq
e^{-dt}
\mathbb E H(X_0,\widehat X_0).
\end{equation*}

The law of
$(X_t,\widehat X_t)
$ is a coupling of
$\mu P_{0,t}^{\eta}$ and $\nu P_{0,t}^{\eta}$. 
Therefore,
\begin{align*}
W_{1,H}
\left(
\mu P_{0,t}^{\eta},
\nu P_{0,t}^{\eta}
\right)
&\leq
\mathbb E H(X_t,\widehat X_t)
\\
&\leq
e^{-dt}
\mathbb E H(X_0,\widehat X_0)
\\
&=
e^{-dt}
\int_{\mathcal X_q\times\mathcal X_q}
H(x,y)\pi(dx,dy).
\end{align*}
Since $\pi\in\Pi(\mu,\nu)$ was arbitrary, taking the infimum over all
couplings of $\mu$ and $\nu$ yields
\begin{equation*}
W_{1,H}
\left(
\mu P_{0,t}^{\eta},
\nu P_{0,t}^{\eta}
\right)
\leq
e^{-dt}
W_{1,H}(\mu,\nu).
\end{equation*}
\end{proof}

\begin{remark}
The estimate is quenched: it compares membrane-potential processes in one
prescribed reinforcement environment. It does not imply contraction of the
full endogenous process $(X,N,S)$, because unmatched spikes generally produce
different future reinforcement profiles.
\end{remark}
\section{Stability with respect to the reinforcement environment}
\label{sec:environment-stability}

The preceding contraction result compares two membrane-potential processes
evolving under the same prescribed reinforcement profile. We now allow the two
processes to evolve under different reinforcement environments and quantify the
resulting perturbation.

Fix $T>0$ and let
\begin{equation*}
\eta_t=(n_t,s_t)
\qquad\text{and}\qquad
\widehat\eta_t=(\widehat n_t,\widehat s_t),
\qquad 0\leq t\leq T,
\end{equation*}
be two deterministic admissible c\`adl\`ag reinforcement profiles.

For every neuron $j$, define
\begin{equation*}
\theta_j^\eta(t)
:=
Q_j^{\eta_t}(+1)
\end{equation*}
and
\begin{equation*}
\theta_j^{\widehat\eta}(t)
:=
Q_j^{\widehat\eta_t}(+1).
\end{equation*}
Equivalently,
\begin{equation}
\theta_j^\eta(t)
=
\frac12
\left(
1+(2p-1)r_j(\eta_t)
\right),
\label{eq:environment-positive-sign-law}
\end{equation}
where
\begin{equation*}
r_j(\eta_t)
:=
\begin{cases}
\dfrac{s_t^j}{n_t^j},&n_t^j\geq1, \\
0,&n_t^j=0.
\end{cases}
\end{equation*}
The analogous definition is used for
$r_j(\widehat\eta_t)$.

For every presynaptic neuron $i$, define the local discrepancy between the two
reinforcement environments by
\begin{equation}
D_i(\eta_t,\widehat\eta_t)
:=
\sum_{j\in\mathcal N_i}
q_j
\left(
w_{ij}+v_{ij}
\right)
\left|
\theta_j^\eta(t)
-
\theta_j^{\widehat\eta}(t)
\right|.
\label{eq:local-environment-discrepancy}
\end{equation}

Since
\begin{equation*}
\theta_j^\eta(t)
-
 \theta_j^{\widehat\eta}(t)
=\frac{2p-1}{2}
\left(
r_j(\eta_t)
-
r_j(\widehat\eta_t)
\right),
\end{equation*}
we also have
\begin{align}
D_i(\eta_t,\widehat\eta_t)
={}&
\frac{|2p-1|}{2}
\sum_{j\in\mathcal N_i}
q_j
\left(
w_{ij}+v_{ij}
\right)
\nonumber 
\left|
r_j(\eta_t)-
r_j(\widehat\eta_t)
\right|.
\label{eq:local-environment-discrepancy-ratio}
\end{align}

\begin{lemma} 
\label{lem:maximal-coupling-different-sign-laws}

Fix a postsynaptic neuron \(j\) and a time \(t\). Let
\begin{equation*}
\theta_j^\eta(t)
:=
Q_j^{\eta_t}(+1),
\qquad
\theta_j^{\widehat\eta}(t)
:=
Q_j^{\widehat\eta_t}(+1).
\end{equation*}
There exists a coupling
$\left(
\varepsilon_j^\eta,
\varepsilon_j^{\widehat\eta}
\right)
$ of the two sign laws
$Q_j^{\eta_t}
$ and $Q_j^{\widehat\eta_t}
$ 
such that
\begin{equation}
\mathbb P
\left(
\varepsilon_j^\eta
\neq
\varepsilon_j^{\widehat\eta}
\right)
=
\left|
\theta_j^\eta(t)
-
\theta_j^{\widehat\eta}(t)
\right|.
\label{eq:sign-mismatch-probability}
\end{equation}

Moreover, define the corresponding synaptic increments by
\begin{equation*}
\xi_{ij}^\eta
=
w_{ij}
\mathbf 1_{\{\varepsilon_j^\eta=+1\}}
-
v_{ij}
\mathbf 1_{\{\varepsilon_j^\eta=-1\}},
\end{equation*}
and
\begin{equation*}
\xi_{ij}^{\widehat\eta}
=
w_{ij}
\mathbf 1_{\{\varepsilon_j^{\widehat\eta}=+1\}}
-
v_{ij}
\mathbf 1_{\{\varepsilon_j^{\widehat\eta}=-1\}}.
\end{equation*}
Then
\begin{equation}
\mathbb E
\left[
\left|
\xi_{ij}^\eta
-
\xi_{ij}^{\widehat\eta}
\right|
\right]
=
\left(
w_{ij}+v_{ij}
\right)
\left|
\theta_j^\eta(t)
-
\theta_j^{\widehat\eta}(t)
\right|.
\label{eq:expected-synaptic-mismatch}
\end{equation}
\end{lemma}

\begin{proof}
The two sign laws are probability measures on the two-point space
$\{-1,+1\}
$. Their probabilities of the value \(+1\) are
\begin{equation*}
Q_j^{\eta_t}(+1)
=
\theta_j^\eta(t),
\qquad
Q_j^{\widehat\eta_t}(+1)
=
\theta_j^{\widehat\eta}(t).
\end{equation*}
Consequently, their probabilities of the value \(-1\) are
\begin{equation*}
Q_j^{\eta_t}(-1)
=
1-\theta_j^\eta(t),
\qquad
Q_j^{\widehat\eta_t}(-1)
=
1-\theta_j^{\widehat\eta}(t).
\end{equation*}

The total variation distance between the two sign laws is therefore
\begin{align*}
\left\|
Q_j^{\eta_t}
-
Q_j^{\widehat\eta_t}
\right\|_{\mathrm{TV}}
&=
\frac12
\sum_{\sigma\in\{-1,+1\}}
\left|
Q_j^{\eta_t}(\sigma)
-
Q_j^{\widehat\eta_t}(\sigma)
\right|
\\
&=
\frac12
\left(
\left|
\theta_j^\eta(t)
-
\theta_j^{\widehat\eta}(t)
\right|
+
\left|
\bigl(1-\theta_j^\eta(t)\bigr)
-
\bigl(1-\theta_j^{\widehat\eta}(t)\bigr)
\right|
\right)
\\
&=
\left|
\theta_j^\eta(t)
-
\theta_j^{\widehat\eta}(t)
\right|.
\end{align*}
Let \(U\) be uniformly distributed on \([0,1]\), and define
\begin{equation*}
\varepsilon_j^\eta
=
\begin{cases}
+1, & U\leq\theta_j^\eta(t),\\
-1, & U>\theta_j^\eta(t),
\end{cases}
\end{equation*}
and
\begin{equation*}
\varepsilon_j^{\widehat\eta}
=
\begin{cases}
+1, & U\leq\theta_j^{\widehat\eta}(t),\\
-1, & U>\theta_j^{\widehat\eta}(t).
\end{cases}
\end{equation*}
Then \(\varepsilon_j^\eta\) and
\(\varepsilon_j^{\widehat\eta}\) have respective laws
\(Q_j^{\eta_t}\) and \(Q_j^{\widehat\eta_t}\). Moreover, they differ exactly
when \(U\) lies between
\(\theta_j^\eta(t)\) and
\(\theta_j^{\widehat\eta}(t)\). Hence,
\begin{equation*}
\mathbb P
\left(
\varepsilon_j^\eta
\neq
\varepsilon_j^{\widehat\eta}
\right)
=
\left|
\theta_j^\eta(t)
-
\theta_j^{\widehat\eta}(t)
\right|.
\end{equation*}
Since this probability equals the total variation distance between the two
sign laws, the coupling is maximal. 
This proves \eqref{eq:sign-mismatch-probability}.

We now compare the corresponding synaptic increments. If
$\varepsilon_j^\eta
=
\varepsilon_j^{\widehat\eta},
$ then 
$\xi_{ij}^\eta
=
\xi_{ij}^{\widehat\eta}$, and hence
\begin{equation*}
\left|
\xi_{ij}^\eta
-
\xi_{ij}^{\widehat\eta}
\right|
=
0.
\end{equation*}

If the signs differ, then one synaptic increment equals \(w_{ij}\), while the
other equals \(-v_{ij}\). Therefore,
\begin{equation*}
\left|
\xi_{ij}^\eta
-
\xi_{ij}^{\widehat\eta}
\right|
=
\left|
w_{ij}-(-v_{ij})
\right|
=
w_{ij}+v_{ij}.
\end{equation*}
Thus,
\begin{equation*}
\left|
\xi_{ij}^\eta
-
\xi_{ij}^{\widehat\eta}
\right|
=
\left(
w_{ij}+v_{ij}
\right)
\mathbf 1_{\{
\varepsilon_j^\eta
\neq
\varepsilon_j^{\widehat\eta}
\}}.
\end{equation*}
Taking expectations and using
\eqref{eq:sign-mismatch-probability}, we obtain
\begin{align*}
\mathbb E
\left[
\left|
\xi_{ij}^\eta
-
\xi_{ij}^{\widehat\eta}
\right|
\right]
=
\left(
w_{ij}+v_{ij}
\right)
\mathbb P
\left(
\varepsilon_j^\eta
\neq
\varepsilon_j^{\widehat\eta}
\right)
=
\left(
w_{ij}+v_{ij}
\right)
\left|
\theta_j^\eta(t)
-
\theta_j^{\widehat\eta}(t)
\right|.
\end{align*}
This proves \eqref{eq:expected-synaptic-mismatch}.
\end{proof}

We now couple two membrane-potential processes
$X_t$
and $\widehat X_t$ 
 evolving under the profiles $\eta$ and $\widehat\eta$, respectively. Their
firing clocks are coupled maximally. At a simultaneous spike of neuron $i$, the
reinforced signs received by each postsynaptic neuron are coupled according to
Lemma~\ref{lem:maximal-coupling-different-sign-laws}. Sign couplings associated
with distinct postsynaptic neurons are conditionally independent.

\begin{lemma} 
\label{lem:perturbed-coupling-generator}

Assume that
\eqref{eq:phi-g-order}--\eqref{eq:general-contraction-rate}
hold. Then the coupling generator associated with the two prescribed
environments \(\eta\) and \(\widehat\eta\) satisfies
\begin{align}
\mathcal L_{2,t}^{\eta,\widehat\eta}H(x,y)
\leq
-dH(x,y)
+
\sum_{i=1}^{\infty}
\min\left\{
\phi_i(x^i),
\phi_i(y^i)
\right\}
D_i(\eta_t,\widehat\eta_t).
\label{eq:perturbed-coupling-generator}
\end{align}
\end{lemma}

\begin{proof}
The deterministic-flow contribution is exactly the same as in
Lemma~\ref{lem:elephant-coupling-generator}, since the reinforcement
environment affects only the synaptic sign distributions and not the
deterministic evolution of the membrane potentials.

We next consider the unmatched firing events. Suppose, for instance, that
neuron $i$ spikes only in the second system. Then the first system is left
unchanged, while the second system resets its $i$-th coordinate and, for every
$j\in\mathcal N_i$, receives a synaptic increment
$\xi_{ij}^{\widehat\eta}$ distributed according to the environment
$\widehat\eta$. For the postsynaptic coordinates we have
\begin{align*}
\left|
x^j-
\left(y^j+\xi_{ij}^{\widehat\eta}\right)_+
\right|
-
|x^j-y^j|
\leq
\left|
\left(y^j+\xi_{ij}^{\widehat\eta}\right)_+
-y^j
\right|
\leq
\left|
\xi_{ij}^{\widehat\eta}
\right|
\leq
c_{ij}.
\end{align*}
Hence the total weighted postsynaptic increase produced by such an unmatched
spike is bounded by
\begin{equation*}
\sum_{j\in\mathcal N_i}q_jc_{ij}
=c_i.
\end{equation*}
This estimate is uniform in the reinforcement environment and therefore does
not require any comparison between the sign laws associated with $\eta$ and
$\widehat\eta$.

The same argument applies to unmatched spikes occurring only in the first
system. Consequently, the deterministic-flow contribution together with all
unmatched firing contributions can be estimated exactly as in
Lemma~\ref{lem:elephant-coupling-generator}. Under assumptions
\eqref{eq:phi-g-order}--\eqref{eq:general-contraction-rate}, these terms produce
the contractive bound
$-dH(x,y).
$

It remains only to estimate the additional contribution arising at
simultaneous spikes, where both systems jump but their synaptic increments may
have different conditional laws because the reinforcement environments
$\eta$ and $\widehat\eta$ are different.

Suppose that neuron \(i\) spikes simultaneously in the two coupled systems.
The \(i\)-th coordinate is reset to zero in both copies. Hence its contribution
to the post-jump distance is zero and, in particular, does not increase the
distance.

For every postsynaptic neuron \(j\in\mathcal N_i\), the corresponding updates
are
\begin{equation*}
x^j
\longmapsto
\left(
x^j+\xi_{ij}^{\eta}
\right)_+,
\end{equation*}
and
\begin{equation*}
y^j
\longmapsto
\left(
y^j+\xi_{ij}^{\widehat\eta}
\right)_+.
\end{equation*}
Since the positive-part map is \(1\)-Lipschitz,
\begin{align*}
&
\left|
\left(
x^j+\xi_{ij}^{\eta}
\right)_+
-
\left(
y^j+\xi_{ij}^{\widehat\eta}
\right)_+
\right|
\leq
\left|
x^j-y^j
+
\xi_{ij}^{\eta}
-
\xi_{ij}^{\widehat\eta}
\right|
\leq
\left|
x^j-y^j
\right|
+
\left|
\xi_{ij}^{\eta}
-
\xi_{ij}^{\widehat\eta}
\right|.
\end{align*}
Therefore, relative to the pre-jump distance at coordinate \(j\), the possible
increase is bounded by
$q_j
\left|
\xi_{ij}^{\eta}
-
\xi_{ij}^{\widehat\eta}
\right|.
$ Taking expectation with respect to the maximal coupling of the two sign laws
and applying
Lemma~\ref{lem:maximal-coupling-different-sign-laws}, we obtain
\begin{align*}
q_j
\mathbb E
\left[
\left|
\xi_{ij}^{\eta}
-
\xi_{ij}^{\widehat\eta}
\right|
\right]
=
q_j
\left(
w_{ij}+v_{ij}
\right)
\left|
\theta_j^\eta(t)
-
\theta_j^{\widehat\eta}(t)
\right|.
\end{align*}
Summing over all \(j\in\mathcal N_i\), define
\begin{equation*}
D_i(\eta_t,\widehat\eta_t)
:=
\sum_{j\in\mathcal N_i}
q_j
\left(
w_{ij}+v_{ij}
\right)
\left|
\theta_j^\eta(t)
-
\theta_j^{\widehat\eta}(t)
\right|.
\end{equation*}
Thus, the additional expected increase of \(H\) caused by a simultaneous spike
of neuron \(i\) is bounded by
$D_i(\eta_t,\widehat\eta_t).
$  Under the maximal coupling of the firing events, simultaneous spikes of neuron
\(i\) occur at rate
$\min\left\{
\phi_i(x^i),
\phi_i(y^i)
\right\}.
$ Therefore, the total additional contribution due to the different sign laws is
bounded by
\begin{equation*}
\sum_{i=1}^{\infty}
\min\left\{
\phi_i(x^i),
\phi_i(y^i)
\right\}
D_i(\eta_t,\widehat\eta_t).
\end{equation*}
Combining this estimate with the contractive contribution
\(-dH(x,y)\) proves
\eqref{eq:perturbed-coupling-generator}.
\end{proof}

The previous lemma gives a general estimate in which the perturbation depends on
the current firing intensities. A deterministic Wasserstein bound follows under
a bounded-rate assumption.

For the remainder of this section only, we impose the additional bounded-rate
assumption that, for every $i\in\mathbb N$, there exists a finite constant
$\overline\phi_i$ such that
\begin{equation}
0\leq
\phi_i(x)
\leq
\overline\phi_i,
\qquad
x\geq0.
\label{eq:bounded-firing-rates-environment}
\end{equation}
Define the global discrepancy between the two reinforcement environments by
\begin{equation}
\mathcal D(\eta_t,\widehat\eta_t)
:=
\sum_{i=1}^{\infty}
\overline\phi_i
D_i(\eta_t,\widehat\eta_t).
\label{eq:global-environment-discrepancy}
\end{equation}
We assume that
\begin{equation}
\int_0^T
\mathcal D(\eta_t,\widehat\eta_t)dt
<\infty.
\label{eq:integrable-environment-discrepancy}
\end{equation}

\begin{theorem} 
\label{thm:wasserstein-different-environments}

Assume that
\eqref{eq:phi-g-order}--\eqref{eq:general-contraction-rate},
\eqref{eq:bounded-firing-rates-environment}, and
\eqref{eq:integrable-environment-discrepancy}
hold. Then, for every pair of probability measures
\(\mu,\nu\in\mathcal P_1(\mathcal X_q)\) with finite first \(H\)-moment,
\begin{align}
W_{1,H}
\left(
\mu P_{0,t}^{\eta},
\nu P_{0,t}^{\widehat\eta}
\right)
\leq
e^{-dt}W_{1,H}(\mu,\nu)
+
\int_0^t
e^{-d(t-u)}
\mathcal D(\eta_u,\widehat\eta_u)
du,
\label{eq:wasserstein-different-environments}
\end{align}
for every \(0\leq t\leq T\).
\end{theorem}

\begin{proof}
Let
$\pi\in\Pi(\mu,\nu)
$ be an arbitrary coupling of the initial laws. Choose
\begin{equation*}
(X_0,\widehat X_0)\sim\pi,
\end{equation*}
and construct the coupled process
$Z_t=(X_t,\widehat X_t)
$ in the two prescribed reinforcement environments
\(\eta\) and \(\widehat\eta\).

By Lemma~\ref{lem:perturbed-coupling-generator},
\begin{align*}
\mathcal L_{2,t}^{\eta,\widehat\eta}H(x,y)
\leq
-dH(x,y)
+
\sum_{i=1}^{\infty}
\min\left\{
\phi_i(x^i),
\phi_i(y^i)
\right\}
D_i(\eta_t,\widehat\eta_t).
\end{align*}
Using \eqref{eq:bounded-firing-rates-environment}, we have
\begin{equation*}
\min\left\{
\phi_i(x^i),
\phi_i(y^i)
\right\}
\leq
\overline\phi_i.
\end{equation*}
Hence,
\begin{align*}
\mathcal L_{2,t}^{\eta,\widehat\eta}H(x,y)
\leq
-dH(x,y)
+
\sum_{i=1}^{\infty}
\overline\phi_i
D_i(\eta_t,\widehat\eta_t)
=
-dH(x,y)
+
\mathcal D(\eta_t,\widehat\eta_t).
\end{align*}

Let \(\tau_R\) be the localization time introduced in
Lemma~\ref{lem:elephant-coupled-integrability}. By the localized
integrability estimate, the extended Dynkin formula may be applied up to
\(\tau_R\).

Define
\begin{equation*}
F(t,z)
:=
e^{dt}H(z).
\end{equation*}
Its time-dependent extended generator satisfies
\begin{align*}
\left(
\partial_t
+
\mathcal L_{2,t}^{\eta,\widehat\eta}
\right)
F(t,z)=
e^{dt}
\left[
dH(z)
+
\mathcal L_{2,t}^{\eta,\widehat\eta}H(z)
\right]
\leq
e^{dt}
\mathcal D(\eta_t,\widehat\eta_t).
\end{align*}

Applying the time-dependent extended Dynkin formula to
\(F(t,Z_t)\), stopped at \(\tau_R\), gives
\begin{align*}
&\mathbb E
\left[
e^{d(t\wedge\tau_R)}
H\bigl(
X_{t\wedge\tau_R},
\widehat X_{t\wedge\tau_R}
\bigr)
\right]
\leq
\mathbb E H(X_0,\widehat X_0)
+
\mathbb E
\int_0^{t\wedge\tau_R}
e^{du}
\mathcal D(\eta_u,\widehat\eta_u)
du.
\end{align*}
Since the environments are prescribed, the function
\(\mathcal D(\eta_u,\widehat\eta_u)\) is deterministic. Therefore,
\begin{align*}
\mathbb E
\left[
e^{d(t\wedge\tau_R)}
H\bigl(
X_{t\wedge\tau_R},
\widehat X_{t\wedge\tau_R}
\bigr)
\right]
\leq
\mathbb E H(X_0,\widehat X_0)
+
\int_0^t
e^{du}
\mathcal D(\eta_u,\widehat\eta_u)
du.
\end{align*}

By non-explosion and the finiteness of the weighted norms on compact time
intervals,
$\tau_R\uparrow T
$ almost surely as \(R\to\infty\). Hence, for every fixed \(t\in[0,T]\),
 $t\wedge\tau_R\longrightarrow t
$ almost surely. Consequently,
\begin{equation*}
e^{d(t\wedge\tau_R)}
H\bigl(
X_{t\wedge\tau_R},
\widehat X_{t\wedge\tau_R}
\bigr)
\longrightarrow
e^{dt}H(X_t,\widehat X_t)
\end{equation*}
almost surely.

Fatou's lemma yields
\begin{align*}
e^{dt}\mathbb E H(X_t,\widehat X_t)\leq
\liminf_{R\to\infty}
\mathbb E
\left[
e^{d(t\wedge\tau_R)}
H\bigl(
X_{t\wedge\tau_R},
\widehat X_{t\wedge\tau_R}
\bigr)
\right]
\leq
\mathbb E H(X_0,\widehat X_0)
+
\int_0^t
e^{du}
\mathcal D(\eta_u,\widehat\eta_u)
du.
\end{align*}
Multiplying by \(e^{-dt}\), we obtain
\begin{align*}
\mathbb E H(X_t,\widehat X_t)
\leq
e^{-dt}
\mathbb E H(X_0,\widehat X_0)
+
\int_0^t
e^{-d(t-u)}
\mathcal D(\eta_u,\widehat\eta_u)
du.
\end{align*}

The law of \((X_t,\widehat X_t)\) is a coupling of
$\mu P_{0,t}^{\eta}
$ and $\nu P_{0,t}^{\widehat\eta}.
$ Therefore,
\begin{align*}
W_{1,H}
\left(
\mu P_{0,t}^{\eta},
\nu P_{0,t}^{\widehat\eta}
\right)
&\leq
\mathbb E H(X_t,\widehat X_t)
\\
\leq
e^{-dt}
&\int_{\mathcal X_q\times\mathcal X_q}
H(x,y)\pi(dx,dy)
+
\int_0^t
e^{-d(t-u)}
\mathcal D(\eta_u,\widehat\eta_u)
du.
\end{align*}
Since \(\pi\in\Pi(\mu,\nu)\) was arbitrary, taking the infimum over all
couplings \(\pi\) gives
\begin{align*}
W_{1,H}
\left(
\mu P_{0,t}^{\eta},
\nu P_{0,t}^{\widehat\eta}
\right)
\leq
e^{-dt}W_{1,H}(\mu,\nu)
+
\int_0^t
e^{-d(t-u)}
\mathcal D(\eta_u,\widehat\eta_u)
du.
\end{align*}
This proves \eqref{eq:wasserstein-different-environments}.
\end{proof}

\begin{corollary} 
\label{cor:recovery-common-environment}
If
$\eta_t=\widehat\eta_t,
 0\leq t\leq T,
$ then
$\mathcal D(\eta_t,\widehat\eta_t)=0,
$ and Theorem~\ref{thm:wasserstein-different-environments} reduces to
\begin{equation*}
W_{1,H}\left(
\mu P_{0,t}^{\eta},
\nu P_{0,t}^{\eta}
\right)
\leq
e^{-dt}W_{1,H}(\mu,\nu).
\end{equation*}
\end{corollary}

\begin{corollary} 
\label{cor:uniform-environment-perturbation}
Suppose that
\begin{equation*}
\mathcal D(\eta_t,\widehat\eta_t)\leq\varepsilon,
\qquad 0\leq t\leq T,
\end{equation*}
for some $\varepsilon\geq0$. Then
\begin{align}
W_{1,H}\left(
\mu P_{0,t}^{\eta},
\nu P_{0,t}^{\widehat\eta}
\right)
\leq
e^{-dt}W_{1,H}(\mu,\nu)
+
\frac{\varepsilon}{d}
\left(
1-e^{-dt}
\right).
\label{eq:uniform-environment-perturbation}
\end{align}
In particular,
\begin{equation*}
\limsup_{t\to\infty}
W_{1,H}\left(
\mu P_{0,t}^{\eta},
\nu P_{0,t}^{\widehat\eta}
\right)
\leq
\frac{\varepsilon}{d},
\end{equation*}
whenever the two environments and the corresponding transition kernels are
defined for all $t\geq0$.
\end{corollary}

\begin{proof}
Use
\begin{equation*}
\int_0^t
e^{-d(t-u)}du
=
\frac{1-e^{-dt}}{d}
\end{equation*}
in \eqref{eq:wasserstein-different-environments}.
\end{proof}

\begin{remark} 
Theorem~\ref{thm:wasserstein-different-environments} separates the effect of the
initial membrane-potential distributions from the effect of the reinforcement
profiles. The first term in
\eqref{eq:wasserstein-different-environments} decays exponentially, while the
second term measures the accumulated discrepancy between the two
excitation--inhibition laws.

The theorem does not establish contraction of the complete endogenous process
$(X,N,S)$. Nevertheless, it shows that the conditional membrane-potential
dynamics depend continuously on the prescribed reinforcement environment.
\end{remark}

\begin{remark} 
Without assumption
\eqref{eq:bounded-firing-rates-environment}, the generator estimate
\eqref{eq:perturbed-coupling-generator} remains valid. In that case one obtains
a perturbation bound involving
\begin{equation*}
\mathbb E\left(
\sum_{i=1}^{\infty}
\min\left\{
\phi_i(X_t^i),
\phi_i(\widehat X_t^i)
\right\}
D_i(\eta_t,\widehat\eta_t)
\right).
\end{equation*}
Additional moment or summability estimates are then required to replace this
quantity by a deterministic function of the two environments.
\end{remark}

\section{Conditional Replica Mean Field equation under the Poisson Hypothesis}
\label{sec:rmf-elephant}
The counters $N_t^i$ are nondecreasing, so the full process
$(X_t,N_t,S_t)$ does not generally possess a stationary law in all
coordinates. We therefore freeze an admissible reinforcement environment
 $\eta=(n,s),
$ where
\begin{equation*}
n_i\in\mathbb N_0,
\qquad
|s_i|\leq n_i,
\qquad
s_i\equiv n_i\pmod 2.
\end{equation*}
In this frozen environment, we derive the stationary Replica Mean Field
equation for the conditional membrane-potential dynamics.

The result of this section is conditional on the Poisson Hypothesis and on
the existence of the corresponding stationary replica laws. In particular,
we do not prove here convergence of the $M$-replica system to a nonlinear
Replica Mean Field process, nor do we establish propagation of chaos.
Rather, assuming the asymptotic independence and convergence properties
specified below, we identify the equation that any limiting stationary
transform must satisfy.

Throughout this section assume
\[
g(x)=x,
\qquad
\phi_i(x)=x+z_i,
\qquad z_i\geq0,
\]
and allow a reset value $r_i\geq0$.  Thus, when neuron $i$ spikes, its own
potential becomes $r_i$.  The case used in the preceding sections is recovered
by taking $r_i=0$.

For each $M\geq2$, consider $M$ replicas of the network.  Write $X_i^m$ for
neuron $i$ in replica $m$.  When neuron $(j,n)$ spikes, then for every
$i\in\mathcal N_j$ it chooses uniformly one target replica
$m\in\{1,\ldots,M\}\setminus\{n\}$ and sends to $(i,m)$ an Elephant increment
whose sign has the frozen law $Q_i^\eta$.  Routing choices for different
postsynaptic neurons are independent.  This is the standard replica
construction: interactions are preserved, but they are redirected uniformly
toward copies in the other replicas.

Fix a neuron $i$ and a replica $m$. For each $M\geq2$, assume that the
corresponding frozen $M$-replica system admits an invariant probability
measure $\pi_M^\eta$. We also assume that $\pi_M^\eta$ has finite exponential
moments in a neighbourhood of the origin for all coordinates appearing in the
transform calculations below. The existence and uniqueness of $\pi_M^\eta$ are not established in the
present section. We assume that the invariant measure $\pi_M^\eta$ is exchangeable with respect
to permutations of the replica labels. In particular, for fixed neurons $i,j$
and a fixed target replica $m$, the joint law of $(X_j^n,X_i^m)$ is the same
for every $n\neq m$.

The analysis in this section is carried out under a set of assumptions different from those used in the global non-explosion results. In particular, we do not require here the summability condition
\begin{align*}
\sum_{i\geq1}\operatorname{Lip}(\phi_i)<\infty.
\end{align*}
Instead, following the finite-dependency approach used for infinite neuronal systems in the Replica Mean Field setting, we assume that for every fixed neuron $i\in\mathbb N$, its presynaptic set
\begin{align*}
  \mathcal P_i
:={j\in\mathbb N:\ i\in\mathcal N_j}
\end{align*}
is finite.
Thus, although the underlying network contains infinitely many neurons, the membrane-potential dynamics of any fixed neuron (i) depend directly on only finitely many presynaptic neurons. The stationary Replica Mean Field calculation below is therefore performed coordinatewise, for each fixed (i), and does not assert the existence of a stationary law for the complete infinite-dimensional reinforced process.

We freeze an admissible reinforcement environment
$
\eta=(n,s)
$ 
and consider the corresponding conditional membrane-potential dynamics. Throughout this section we take
\begin{align*}
g(x)=x,
\phi_i(x)=x+z_i,
\qquad z_i\geq0.
\end{align*}
These linear firing rates need not satisfy the global Lipschitz-summability assumption imposed in the non-explosion section; this is not required here because, for every fixed neuron (i), all RMF interaction terms involve only the finite set $\mathcal P_i$.

The reinforcement counters are nondecreasing, so stationarity is imposed only on the conditional membrane-potential dynamics in the frozen environment. The result obtained below is therefore a coordinatewise conditional stationary RMF identity rather than a stationary result for the full process ((X,N,S)).

 All statements below are therefore conditional on the
existence of such invariant probability measures satisfying the required
integrability properties.  Define
\[
\Lambda_{i}^{M,m,\eta}(u)
:=
\mathbb E_{\pi_M^\eta}
\left[e^{u(X_i^m+z_i)}\right]
\]
and
\[
\beta_j^M
:=
\mathbb E_{\pi_M^\eta}[X_j^n+z_j],
\]
which is independent of $n$ by replica symmetry.

For $v\geq0$, also define
\[
\Theta_i^{M,m,\eta}(u,v)
:=
\mathbb E_{\pi_M^\eta}
\left[
 e^{u((X_i^m-v)_++z_i)}
\right].
\]

We use the following Poisson Hypothesis.  For $m\neq n$, as $M\to\infty$, the
state of the presynaptic neuron $(j,n)$ becomes asymptotically independent of
the state of the target neuron $(i,m)$, and the empirical incoming activity
converges to a Poisson input of rate
\[
\beta_j
:=
\lim_{M\to\infty}\beta_j^M.
\]
For \(i\in\mathbb N\), define
\begin{equation*}
\Lambda_i^\eta(u)
:=
\mathbb E_\eta
\left[
e^{u(X_i+z_i)}
\right],
\end{equation*}
and, for \(v\geq0\),
\begin{equation*}
\Theta_i^\eta(u,v)
:=
\mathbb E_\eta
\left[
e^{u((X_i-v)_++z_i)}
\right].
\end{equation*}

\begin{lemma} 
\label{lemnew}
Let $I\subset\mathbb R$ be an open interval containing $0$. Fix
$i\in\mathbb N$ and a replica $m$. Assume that, for every $u\in I$,
\begin{align*}
\Lambda_i^{M,m,\eta}(u)
=
\mathbb E_{\pi_M^\eta}
\left[
e^{u(X_i^m+z_i)}
\right]
\longrightarrow
\Lambda_i^\eta(u)
=
\mathbb E_{\eta}
\left[
e^{u(X_i+z_i)}
\right].
\end{align*}

Assume moreover that the family of replica laws satisfies a uniform
exponential-moment condition on compact subsets of $I$. More precisely, for
every compact set $K\subset I$, there exists $\delta_K>0$ such that
$u+2\delta_K\in I$ for every $u\in K$ and
\begin{align*}
\sup_{M\geq2}
\sup_{u\in K}
\mathbb E_{\pi_M^\eta}
\left[
e^{(u+2\delta_K)(X_i^m+z_i)}
\right]
<\infty.
\end{align*}

Then, for every $u\in I$,
\begin{align*}
\bigl(\Lambda_i^{M,m,\eta}\bigr)'(u)
\longrightarrow
\bigl(\Lambda_i^\eta\bigr)'(u).
\end{align*}

In particular,
\begin{align*}
\mathbb E_{\pi_M^\eta}
\left[
(X_i^m+z_i)e^{u(X_i^m+z_i)}
\right]
\longrightarrow
\mathbb E_{\eta}
\left[
(X_i+z_i)e^{u(X_i+z_i)}
\right].
\end{align*}

\end{lemma}

\begin{proof}

Set
\begin{align*}
Y_M:=X_i^m+z_i\ \text{and}\
Y:=X_i+z_i.
\end{align*}

By assumption,
\begin{align*}
\mathbb E_{\pi_M^\eta}
\left[
e^{uY_M}
\right]
\longrightarrow
\mathbb E_{\eta}
\left[
e^{uY}
\right],
\qquad
u\in I.
\end{align*}

Since the moment generating functions converge on an open interval containing
$0$ to the moment generating function of $Y$, it follows that
$Y_M\Longrightarrow Y.
$ Fix $u\in I$. Choose $\delta>0$ such that $u+2\delta\in I$ and
\begin{align*}
\sup_{M\geq2}
\mathbb E_{\pi_M^\eta}
\left[
e^{(u+2\delta)Y_M}
\right]
<\infty.
\end{align*}

Since $Y_M\geq0$, for every $A>0$,
\begin{align*}
Y_Me^{uY_M}\mathbf 1_{{Y_M>A}}
&=
Y_Me^{-\delta Y_M}
e^{(u+\delta)Y_M}
\mathbf 1_{{Y_M>A}}.
\end{align*}

Since
\begin{align*}
xe^{-\delta x}
\leq
\frac{1}{e\delta},
\qquad
x\geq0,
\end{align*}
we obtain
\begin{align*}
Y_Me^{uY_M}\mathbf 1_{{Y_M>A}}
&\leq
\frac{1}{e\delta}
e^{(u+\delta)Y_M}
\mathbf 1_{{Y_M>A}}.
\end{align*}

On the event ${Y_M>A}$,
\begin{align*}
e^{(u+\delta)Y_M}
=
e^{-\delta Y_M}
e^{(u+2\delta)Y_M}
\leq
e^{-\delta A}
e^{(u+2\delta)Y_M}.
\end{align*}

Consequently,
\begin{align*}
\mathbb E_{\pi_M^\eta}
\left[
Y_Me^{uY_M}\mathbf 1_{{Y_M>A}}
\right]
&\leq
\frac{e^{-\delta A}}{e\delta}
\mathbb E_{\pi_M^\eta}
\left[
e^{(u+2\delta)Y_M}
\right].
\end{align*}

Taking the supremum over $M$ gives
\begin{align*}
\sup_{M\geq2}
\mathbb E_{\pi_M^\eta}
\left[
Y_Me^{uY_M}\mathbf 1_{{Y_M>A}}
\right]
&\leq
\frac{e^{-\delta A}}{e\delta}
\sup_{M\geq2}
\mathbb E_{\pi_M^\eta}
\left[
e^{(u+2\delta)Y_M}
\right].
\end{align*}

Hence,
\begin{align*}
\lim_{A\to\infty}
\sup_{M\geq2}
\mathbb E_{\pi_M^\eta}
\left[
Y_Me^{uY_M}\mathbf 1_{{Y_M>A}}
\right]
=
0.
\end{align*}

Thus, the family
${Y_Me^{uY_M}:M\geq2}$
is uniformly integrable.

Since the map $y\mapsto ye^{uy}$ is continuous on $\mathbb R_+$, weak
convergence together with uniform integrability gives
\begin{align*}
\mathbb E_{\pi_M^\eta}
\left[
Y_Me^{uY_M}
\right]
\longrightarrow
\mathbb E_{\eta}
\left[
Ye^{uY}
\right].
\end{align*}

Moreover,
\begin{align*}
\bigl(\Lambda_i^{M,m,\eta}\bigr)'(u)=
\mathbb E_{\pi_M^\eta}
\left[
Y_Me^{uY_M}
\right],
\bigl(\Lambda_i^\eta\bigr)'(u)
=
\mathbb E_{\eta}
\left[
Ye^{uY}
\right].
\end{align*}

Therefore,
\begin{align*}
\bigl(\Lambda_i^{M,m,\eta}\bigr)'(u)
\longrightarrow
\bigl(\Lambda_i^\eta\bigr)'(u).
\end{align*}

\end{proof}
 
\begin{proposition} 
\label{prop:conditional-rmf-elephant}

Fix a time-independent reinforcement environment $\eta$. Assume that, for
every $M\geq 2$, the frozen $M$-replica process admits an invariant probability
measure $\pi_M^\eta$ with the exponential integrability required below.
Assume moreover that $\pi_M^\eta$ is exchangeable with respect to
permutations of the replica labels.

Assume further that the Poisson Hypothesis holds.

Let
\begin{align*}
\beta_i
&:=
\mathbb E_\eta[X_i+z_i].
\end{align*}

Assume moreover that, as $M\to\infty$,
\begin{align*}
\beta_i^M
&\longrightarrow
\beta_i,
\\
\Lambda_i^{M,m,\eta}(u)
&\longrightarrow
\Lambda_i^\eta(u),
\\
\Theta_i^{M,m,\eta}(u,v)
&\longrightarrow
\Theta_i^\eta(u,v),
\end{align*}
for $u$ in a common open interval $I$ containing $0$.

Assume in addition that, for every compact set $K\subset I$, there exists
$\delta_K>0$ such that $u+2\delta_K\in I$ for every $u\in K$ and
\begin{align*}
\sup_{M\geq 2}
\sup_{u\in K}
\mathbb E_{\pi_M^\eta}
\left[
e^{(u+2\delta_K)(X_i^m+z_i)}
\right]
&<\infty.
\end{align*}

Then, for every $u\in I$,
\begin{align}
\label{eqrmf}0
=&
-(1+a_i u)
\bigl(\Lambda_i^\eta\bigr)'(u)
+
a_i u z_i\Lambda_i^\eta(u)
+
\beta_i e^{u(r_i+z_i)}+
\\
\nonumber +
&\sum_{j\in\mathcal P_i}
\beta_j Q_i^\eta(+1)
\bigl(e^{u w_{ji}}-1\bigr)
\Lambda_i^\eta(u)
+
\sum_{j\in\mathcal P_i}
\beta_j Q_i^\eta(-1)
\left[
\Theta_i^\eta(u,v_{ji})
-
\Lambda_i^\eta(u)
\right].
\end{align}

Moreover,
\begin{align*}
\beta_i
&=
\bigl(\Lambda_i^\eta\bigr)'(0).
\end{align*}

If $r_i=0$, the contribution arising from the reset mechanism before
combination with the deterministic drift is
\begin{align*}
\beta_i e^{u z_i}
-
\bigl(\Lambda_i^\eta\bigr)'(u).
\end{align*}

\end{proposition}

\begin{proof}

For a fixed replica $m$, consider the test function
\begin{align*}
f_u(x)
:=
e^{u(x_i^m+z_i)}.
\end{align*}

Only the deterministic drift of $x_i^m$, the spikes of neuron $(i,m)$,
and incoming routed spikes affect $f_u$. The corresponding part of the
$M$-replica generator is
\begin{align*}
\mathcal G_M^\eta f_u(x)
={}&
-a_i x_i^m u f_u(x)
+
(x_i^m+z_i)
\left[
e^{u(r_i-x_i^m)}-1
\right]f_u(x)
\\
&+
\sum_{j\in\mathcal P_i}
\sum_{n\neq m}
\frac{x_j^n+z_j}{M-1}
Q_i^\eta(+1)
\bigl(e^{u w_{ji}}-1\bigr)f_u(x)
\\
&+
\sum_{j\in\mathcal P_i}
\sum_{n\neq m}
\frac{x_j^n+z_j}{M-1}
Q_i^\eta(-1)
\left[
e^{u((x_i^m-v_{ji})_++z_i)}
-
e^{u(x_i^m+z_i)}
\right].
\end{align*}

Since $\pi_M^\eta$ is invariant and $f_u$ belongs to the domain of the
extended generator,
\begin{align*}
\mathbb E_{\pi_M^\eta}
\left[
\mathcal G_M^\eta f_u
\right]
=
0.
\end{align*}

Define
\begin{align*}
\Lambda_i^{M,m,\eta}(u)
:=
\mathbb E_{\pi_M^\eta}
\left[
e^{u(X_i^m+z_i)}
\right].
\end{align*}

Differentiating under the expectation gives
\begin{align*}
\bigl(\Lambda_i^{M,m,\eta}\bigr)'(u)
=
\mathbb E_{\pi_M^\eta}
\left[
(X_i^m+z_i)e^{u(X_i^m+z_i)}
\right].
\end{align*}

The drift contribution is
\begin{align*}
-a_i u
\mathbb E_{\pi_M^\eta}
\left[
X_i^m e^{u(X_i^m+z_i)}
\right]
&=
-a_i u
\left[
\bigl(\Lambda_i^{M,m,\eta}\bigr)'(u)
-
z_i\Lambda_i^{M,m,\eta}(u)
\right].
\end{align*}

For the reset contribution, observe that
\begin{align*}
(x_i^m+z_i)
\left[
e^{u(r_i-x_i^m)}-1
\right]
e^{u(x_i^m+z_i)}
=
(x_i^m+z_i)
\left[
e^{u(r_i+z_i)}
-
e^{u(x_i^m+z_i)}
\right].
\end{align*}

Therefore,
\begin{align*}
\mathbb E_{\pi_M^\eta}
\left[
(X_i^m+z_i)
\left(
e^{u(r_i+z_i)}
-
e^{u(X_i^m+z_i)}
\right)
\right]
=
\beta_i^M e^{u(r_i+z_i)}
-
\bigl(\Lambda_i^{M,m,\eta}\bigr)'(u),
\end{align*}
where
\begin{align*}
\beta_i^M
:=
\mathbb E_{\pi_M^\eta}[X_i^m+z_i].
\end{align*}

Combining the deterministic drift and reset contributions gives
\begin{align*}
&-a_i u
\left[
\bigl(\Lambda_i^{M,m,\eta}\bigr)'(u)
-
z_i\Lambda_i^{M,m,\eta}(u)
\right]
+
\beta_i^M e^{u(r_i+z_i)}
-
\bigl(\Lambda_i^{M,m,\eta}\bigr)'(u)
\\
&\qquad=
-(1+a_i u)
\bigl(\Lambda_i^{M,m,\eta}\bigr)'(u)
+
a_i u z_i\Lambda_i^{M,m,\eta}(u)
+
\beta_i^M e^{u(r_i+z_i)}.
\end{align*}

By Lemma~\ref{lemnew},
\begin{align*}
\bigl(\Lambda_i^{M,m,\eta}\bigr)'(u)
\longrightarrow
\bigl(\Lambda_i^\eta\bigr)'(u),
\qquad
u\in I.
\end{align*}

Hence, together with the assumed convergence of
$\Lambda_i^{M,m,\eta}(u)$ and $\beta_i^M$, the deterministic drift and reset
contribution converges to
\begin{align*}
&-(1+a_i u)
\bigl(\Lambda_i^\eta\bigr)'(u)
+
a_i u z_i\Lambda_i^\eta(u)
+
\beta_i e^{u(r_i+z_i)}.
\end{align*}
For an excitatory incoming spike, the Poisson Hypothesis and asymptotic
independence of distinct replicas yield, for any $n\neq m$,
\begin{align*}
\mathbb E_{\pi_M^\eta}
\left[
(X_j^n+z_j)e^{u(X_i^m+z_i)}
\right]
&\longrightarrow
\beta_j\Lambda_i^\eta(u).
\end{align*}

By exchangeability of $\pi_M^\eta$ with respect to the replica labels, for
fixed $m$ the quantity
\begin{align*}
\mathbb E_{\pi_M^\eta}
\left[
(X_j^n+z_j)e^{u(X_i^m+z_i)}
\right]
\end{align*}
is independent of the choice of $n\neq m$. Hence, for any fixed $n_0\neq m$,
\begin{align*}
\frac{1}{M-1}
\sum_{n\neq m}
\mathbb E_{\pi_M^\eta}
\left[
(X_j^n+z_j)e^{u(X_i^m+z_i)}
\right]
=
\mathbb E_{\pi_M^\eta}
\left[
(X_j^{n_0}+z_j)e^{u(X_i^m+z_i)}
\right].
\end{align*}

Therefore,
\begin{align*}
\frac{1}{M-1}
\sum_{n\neq m}
\mathbb E_{\pi_M^\eta}
\left[
(X_j^n+z_j)e^{u(X_i^m+z_i)}
\right]
\longrightarrow
\beta_j\Lambda_i^\eta(u).
\end{align*}

Consequently, the excitatory incoming contribution converges to
\begin{align*}
\sum_{j\in\mathcal P_i}
\beta_j Q_i^\eta(+1)
\bigl(e^{u w_{ji}}-1\bigr)
\Lambda_i^\eta(u).
\end{align*}
Similarly, for any $n\neq m$,
\begin{align*}
\mathbb E_{\pi_M^\eta}
\left[
(X_j^n+z_j)
e^{u((X_i^m-v_{ji})_++z_i)}
\right]
&\longrightarrow
\beta_j\Theta_i^\eta(u,v_{ji}).
\end{align*}

By replica exchangeability, the expectation on the left-hand side is
independent of the choice of $n\neq m$. Hence, for any fixed $n_0\neq m$,
\begin{align*}
\frac{1}{M-1}
\sum_{n\neq m}
\mathbb E_{\pi_M^\eta}
\left[
(X_j^n+z_j)
e^{u((X_i^m-v_{ji})_++z_i)}
\right]
&=
\mathbb E_{\pi_M^\eta}
\left[
(X_j^{n_0}+z_j)
e^{u((X_i^m-v_{ji})_++z_i)}
\right]
\\
&\longrightarrow
\beta_j\Theta_i^\eta(u,v_{ji}).
\end{align*}

Likewise,
\begin{align*}
\frac{1}{M-1}
\sum_{n\neq m}
\mathbb E_{\pi_M^\eta}
\left[
(X_j^n+z_j)e^{u(X_i^m+z_i)}
\right]
\longrightarrow
\beta_j\Lambda_i^\eta(u).
\end{align*}

Therefore, the inhibitory incoming contribution converges to
\begin{align*}
\sum_{j\in\mathcal P_i}
\beta_j Q_i^\eta(-1)
\left[
\Theta_i^\eta(u,v_{ji})
-
\Lambda_i^\eta(u)
\right].
\end{align*}

Since $\mathcal P_i$ is finite, the limit may be taken term by term in the
sum over $j\in\mathcal P_i$. Combining the deterministic drift, reset, excitatory, and
inhibitory contributions therefore yields
(\ref{eqrmf}).

Finally, by definition,
\begin{align*}
\Lambda_i^\eta(u)
&=
\mathbb E_\eta
\left[
e^{u(X_i+z_i)}
\right].
\end{align*}

Since the limiting stationary law has finite exponential moments on the
interval $I$, differentiation under the expectation is justified for
$u\in I$. Hence,
\begin{align*}
\bigl(\Lambda_i^\eta\bigr)'(u)
&=
\mathbb E_\eta
\left[
(X_i+z_i)e^{u(X_i+z_i)}
\right].
\end{align*}

Evaluating at $u=0$ gives
\begin{align*}
\bigl(\Lambda_i^\eta\bigr)'(0)
=
\mathbb E_\eta
\left[
X_i+z_i
\right]
 =
\beta_i.
\end{align*}

\end{proof}

\begin{remark} 
\label{rem:rmf-invariant-measure}
The existence of the invariant measures $\pi_M^\eta$ is an assumption of
Proposition~\ref{prop:conditional-rmf-elephant}. The proposition does not
provide an existence or uniqueness theorem for stationary laws of the frozen
replica system. Its conclusion is an identity satisfied by any limiting
stationary transform whenever the invariant measures and the stated limiting
properties exist.
\end{remark}
\begin{remark} 
\label{rem:scope-rmf-result}
Proposition~\ref{prop:conditional-rmf-elephant} is a conditional identification
result. It assumes the existence of stationary replica laws, the Poisson
Hypothesis, and the convergence of the relevant transforms and mean firing
rates. Under these assumptions, it identifies the equation satisfied by the
limiting stationary transform.

No propagation-of-chaos theorem or convergence theorem for the full
$M$-replica process is proved in the present work. Thus,
\eqref{eqrmf} should be understood as the conditional
stationary RMF identity associated with the frozen reinforcement environment.
\end{remark}
\begin{remark} 
The equation is not closed in $\Lambda_i^\eta$ alone.  Indeed,
\[
\Theta_i^\eta(u,v)
=
 e^{uz_i}\mathbb P_\eta(X_i\leq v)
+
 e^{-uv}
\mathbb E_\eta
\left[
 e^{u(X_i+z_i)}\mathbf 1_{\{X_i>v\}}
\right].
\]
The additional term is forced by the boundary truncation
$x\mapsto(x-v)_+$ and contains information about the distribution below the
inhibitory threshold.
\end{remark}

\begin{remark}
The closed equation obtained by replacing $e^{uw_{ji}}$ with
$\Psi_{ji}^\eta(u)$ is therefore exact only for the additive model.  For the
original model on $\mathbb R_+$, the correct RMF equation is
\eqref{eqrmf}.
\end{remark}

\end{document}